\documentclass[11pt,a4paper]{article}
\usepackage[top=1in, bottom=1.25in, left=1in, right=1in]{geometry}
\usepackage[latin1]{inputenc}
\usepackage[english]{babel}
\usepackage{graphicx,epstopdf}
\usepackage{caption, subcaption,float}
\usepackage{amsthm}
\usepackage{color}
\usepackage{tikz}
\usepackage{relsize}
\usepackage[ruled, linesnumberedhidden,vlined]{algorithm2e} 
\usepackage[round,sort,comma]{natbib}
\usepackage{aliascnt}
\usepackage{hyperref}
\hypersetup{
    colorlinks=true,       
    linkcolor=blue,          
    citecolor=olive,       
    filecolor=magenta,      
    urlcolor=blue          
}

\usetikzlibrary{matrix}
\usepackage{amsmath,amssymb,dsfont}
\begin{document}

\title{Commensurability in Artin groups of spherical type}
\date{\today }
\author{Mar\'{i}a Cumplido and Luis Paris\\  Universit\'{e} de Bourgogne}

\maketitle
\newtheorem{theorem}{Theorem}

\newaliascnt{lemma}{theorem}
\newtheorem{lemma}[lemma]{Lemma}
\aliascntresetthe{lemma}
\providecommand*{\lemmaautorefname}{Lemma}

\newaliascnt{proposition}{theorem}
\newtheorem{proposition}[proposition]{Proposition}
\aliascntresetthe{proposition}
\providecommand*{\propositionautorefname}{Proposition}

\newaliascnt{corollary}{theorem}
\newtheorem{corollary}[corollary]{Corollary}
\aliascntresetthe{corollary}
\providecommand*{\corollaryautorefname}{Corollary}

\theoremstyle{definition}
\newtheorem{definition}[theorem]{Definition}

\theoremstyle{remark}
\newtheorem{remark}[theorem]{Remark}

\newtheorem{claim}{Claim}
\providecommand*{\claimautorefname}{Claim}

\def\N{\mathbb N} \def\SSS{\mathfrak S} \def\Z{\mathbb Z}
\def\Com{{\rm Com}} \def\CCom{{\widetilde{\Com}}} \def\CA{{\rm CA}}
\def\Ker{{\rm Ker}} \def\R{{\mathbb R}} \def\GL{{\rm GL}} 
\def\HH{\mathcal H} \def\C{\mathbb C} \def\P{\mathbb P}
\def\QZ{{\rm QZ}} \def\PP{\mathcal P} \def\MM{\mathcal M}
\def\BB{\mathcal B} \def\Im{{\rm Im}} \def\ord{{\rm ord}}
\def\id{{\rm id}} \def\CC{\mathcal C} \def\S{{\mathbb S}}
\def\AA{\mathcal A} \def\D{\mathbb D} \def\SS{\mathcal S}
\def\ordre{{\rm ordre}}

\newcommand{\myref}[2]{\hyperref[#1]{#2~\ref*{#1}}}

\begin{abstract} We give
an almost complete classification of Artin groups of spherical type up
to commensurability.
Let $A$ and $A'$ be two Artin groups of spherical type, and let $A_1,\dots,A_p$ (resp. $A'_1,\dots,A'_q$) be the irreducible components of $A$ (resp. $A'$). We show that $A$ and $A'$ are commensurable if and only if $p=q$ and, up to permutation of the indices, $A_i$ and $A'_i$ are commensurable for every $i$. We prove that, if two Artin groups of spherical type are commensurable, then they have the same rank. For a fixed $n$, we give a complete classification of the irreducible Artin groups of rank $n$ that are commensurable with the group of type $A_n$. Note that there are 6 remaining comparisons of pairs of groups to get the complete classification of Artin groups of spherical type up to commensurability, two of which have been done by Ignat Soroko after the first version of the present paper.


\end{abstract}

\section{Introduction}

We start by recalling the definitions of Coxeter groups and Artin groups.
Let~$S$ be a finite set.
A \emph{Coxeter matrix} over~$S$ is a square matrix $M = (m_{s,t})_{s,t \in S}$ indexed by the elements of~$S$, having coefficients in~$\N \cup \{ \infty \}$, and satisfying $m_{s,s} = 1$ for every $s \in S$,
and $m_{s,t} = m_{t,s} \ge 2$ for every $s,t \in S$, $s \neq t$.
This matrix is represented by a labeled graph~$\Gamma$, called \emph{Coxeter graph} and defined by the following data.
The set of vertices of~$\Gamma$ is~$S$.
Two vertices $s,t \in S$, $s \neq t$, are connected by an edge if $m_{s,t} \ge 3$, and this edge is labeled with~$m_{s,t}$ if $m_{s,t} \ge 4$.

\bigskip\noindent
If $s,t \in S$ and~$m$ is an integer $\ge 2$, we denote by $\Pi (s,t, m)$ the word $sts \cdots$ of length~$m$.
In other words, $\Pi (s,t,m) = (st)^{\frac{m}{2}}$  if~$m$ is even and $\Pi (s,t,m) = (st)^{\frac{m-1}{2}}s$ if~$m$ is odd.
Let~$\Gamma$ be the Coxeter graph associated to such a Coxeter matrix. 
The \emph{Artin group} associated to~$\Gamma$ is the group~$A =A[\Gamma]$ defined by the following presentation.
\[
A[\Gamma] = \langle S \mid \Pi (s,t, m_{s,t}) = \Pi (t,s, m_{s,t}), \text{ for } s,t \in S,\ s \neq t,\ m_{s,t} \neq \infty \rangle\,.
\]
The \emph{Coxeter group}~$W = W[\Gamma]$ of~$\Gamma$ is the quotient of~$A[\Gamma]$ by the relations $s^2=1$, $s \in S$.
We say that $\Gamma$ is of \emph{spherical type} if~$W[\Gamma]$ is finite.

\bigskip\noindent
Let $\Gamma_1, \dots, \Gamma_p$ be the connected components of~$\Gamma$ and, for $i \in \{1, \dots, p\}$, let $S_i$ be the set of vertices of~$\Gamma_i$, $A_i$ be the subgroup of~$A$ generated by~$S_i$ and~$W_i$ be the subgroup of~$W$ generated by~$S_i$.
We can easily check that~$A_i$ is the Artin group of~$\Gamma_i$ and $W_i$ is the Coxeter group of~$\Gamma_i$ for every~$i$, and that $A= A_1 \times \cdots \times A_p$ and $W = W_1 \times \cdots \times W_p$.
In particular, $\Gamma$ has spherical type if and only if~$\Gamma_i$ has spherical type for every $i \in \{ 1, \dots, p\}$. 
The classification of Coxeter graphs of spherical type has been known for a long time and it is given in the following theorem:

\begin{theorem}[\citealp{Coxeter}]
A Coxeter graph~$\Gamma$ is connected and has spherical type if and only if it is isomorphic to one of the graphs $A_n$ ($n \ge 1$), $B_n$ ($n \ge 2$), $D_n$ ($n \ge 4$), $E_n$ ($n \in \{6,7,8\}$), $F_4$, $H_3$, $H_4$ and $I_2 (p)$ ($p \ge 5$) represented in \autoref{coxeter}.
\end{theorem}

\bigskip
\begin{figure}
\centering
\includegraphics[width=5.8cm]{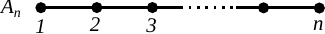}\hskip1cm
\includegraphics[width=5.8cm]{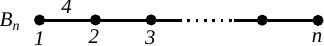}
\vskip 0.5cm
\includegraphics[width=5.8cm]{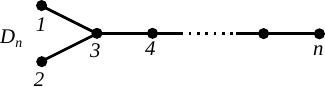}\hskip1cm
\includegraphics[width=4.8cm]{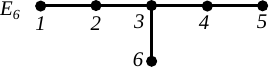}
\vskip 0.5cm
\includegraphics[width=5.8cm]{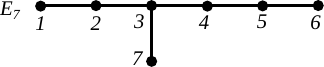}\hskip1cm
\includegraphics[width=6.8cm]{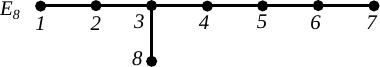}
\vskip 0.5cm
\includegraphics[width=3.8cm]{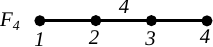}\hskip1cm
\includegraphics[width=2.8cm]{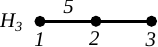}
\vskip 0.5cm
\includegraphics[width=3.8cm]{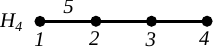}\hskip1cm
\includegraphics[width=2.6cm]{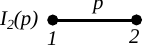}
\vskip 0.5cm
\caption{Coxeter graphs of spherical type}
\label{coxeter}
\end{figure}

\medskip\noindent
Actually, this classification is also the classification of Artin groups of spherical type up to isomorphism because, by \citep[Theorem~1.1]{Paris1}, two Artin groups of spherical type are isomorphic if and only if their associated Coxeter graphs are isomorphic. It is then natural to ask if such a result remains valid when changing the word ``isomorphic'' by ``commensurable''. The answer has been known for a long time: it is \textbf{no} because the Artin groups associated to $A_n$ and $B_n$ are commensurable (see \autoref{AnBn}) and they are not isomorphic by \citep[Theorem~1.1]{Paris1}. However, the classification of Artin groups of spherical type up to commensurability was a very open question before this article. For instance, no example of two non-commensurable Artin groups of spherical type having the same rank was known before.  This article almost gives the entire classification of Artin groups of spherical type up to commensurability, meaning that there are only 6 comparisons of groups that we do not treat. Two of them have been solved by \cite{Soroko} after the first version of the present paper.

\medskip
\noindent
We recall that two groups $G_1$ and $G_2$ are \emph{commensurable} if there are two finite index subgroups~$H_1$ of~$G_1$ and~$H_2$ of~$G_2$ such that~$H_1$ is isomorphic to~$H_2$. The study of commensurability is useful when studying virtual properties of groups. There is also a strong relationship between commensurable groups and quasi-isometric groups. In particular, for a finitely generated group~$G$ endowed with any word metric, the inclusion map of a finite index subgroup in $G$ is a quasi-isometry. This implies that, if two finitely generated groups are commensurable, then they are also quasi-isometric. The converse implication is true only under certain conditions.  

\medskip\noindent
The \emph{commensurator} (also called abstract commensurator) of a group~$G$ will be denoted by $\Com (G)$. We recall its definition.
Let~$\CCom (G)$ be the set of triples~$(U,V,f)$ where~$U$ and~$V$ are finite index subgroups of~$G$, and $f : U \to V$ is an isomorphism.
Let~$\sim$ be the equivalence relation on~$\CCom (G)$ such that $(U,V,f) \sim (U', V', f')$ if there is a finite index subgroup~$W$ of~$U \cap U'$ such that $f(\alpha) = f' (\alpha)$ for every $\alpha \in W$.
Hence we define $\Com (G)$ as $\CCom (G)/ \sim$ and the group operation is induced by the composition. We can easily show that, if~$A$ and~$B$ are two commensurable groups, then~$\Com (A)$ and~$\Com (B)$ are isomorphic. Commensurators are in general difficult to compute. Fortunately, the commensurator of the Artin group associated to~$A_n$ (the braid group) is well understood \citep{CharneyCrisp, LM} and it is indeed used to prove the results in this paper.

\medskip\noindent
So far, the results regarding commensurability for Artin groups in general are quite limited. In \citep{Crisp}, the author studies commensurability for Artin groups of large type (each $m_{s,t}\geq 3$ for $s\neq t$) associated to triangle-free connected Coxeter graphs having at least three vertices. In the last years, the research on this topic has been focused on right-angled Artin groups (RAAGs). A RAAG is an Artin group whose only relations in its presentation are commutations. It is often represented by a \emph{commutation graph}, $\Upsilon$, which is defined by the following data. The set of vertices of~$\Upsilon$ is the set of standard generators of the group.
Two vertices are connected by an edge if and only if the corresponding generators commute. Apart from the classifications made for free and free-abelian groups, commensurability studies are made for RAAGs with commutation graphs $\Upsilon$ in the following cases: 

\begin{itemize}

\item $\Upsilon$ is connected, triangle-free and square-free without any vertices of degree one \citep{KK};

\item  $\Upsilon$  is star-rigid with no induced 4-cycles and the outer
automorphism of the Artin group is finite \citep{Juan};

\item  $\Upsilon$ is a tree of diameter $\leq 4$ \citep{BN, CKZ};

\item $\Upsilon$ is a path graph \citep{CKZ2}. In this work they also compared these commensurability classes to the ones of RAAGs defined by trees of diameter 4.

\end{itemize}

\medskip
\noindent\emph{Remark.}
The results of this paper, notably Part~(3) of \autoref{classification}, are being used in a paper in preparation of Ursula Hamenst\"adt \citep{Ursula} to refute a conjecture made by Kontsevich and Zorich \citep{K}. We fix a tuple of non-negative integers $d=(p_1, p_2,\dots, p_k)$ and consider the vector space of  holomorphic one-forms of a Riemann surface with genus $g$ bigger or equal to 2. We denote by~$M_d$  the moduli space of  these one-forms having zeros $x_1, x_2,\dots ,x_k$ with multiplicity $p_1, p_2,\dots, p_k$, respectively. The conjecture says that each connected component of $M_d$ has homotopy type $K(G, 1)$,  where~$G$ is a group commensurable to some mapping class group. Hamenst\"adt uses the results in \citep{lomo}, to show that there are components in genus~3
that are classifying spaces for the quotients of the Artin groups
$A[E_6]$ and $A[E_7]$ by their centers. She proves that the
only mapping class group which could be commensurable to $A[E_6]/Z(A[E_6])$ is the quotient of the braid group on 7 strands by its center, that is, $A[A_6]/Z(A[A_6])$. By \autoref{tecnico}, the non-commensurability of $A[E_6]/Z(A[E_6])$ and $A[A_6]/Z(A[A_6])$ is equivalent to the non-commensurability of $A[E_6]$ and $A[A_6]$. These
components provide a counterexample to the conjecture. 

\bigskip

\noindent
\Large{\bf Acknowledgements.}

\normalsize The first author was financed by a postdoctoral fellowship of the University of Burgundy and supported by the research grants MTM2016-76453-C2-1-P (financed by the Spanish Ministry of Economy and FEDER) and US-1263032 (financed by the Andalusian Ministry of Economy and Knowledge and the Operational Program FEDER 2014--2020). The second author was supported by the French project ``AlMaRe" (ANR-19-CE40-0001-01) of the ANR. Both authors also thank the two referees for their carreful reading of the manuscript and for their interesting comments and
suggestions.

\section{Statements}
\noindent
Recall that our aim is to partially classify the Artin groups of spherical type up to commensurability. Our starting point is the following result which can be easily proven. It allows to reduce the question to the case where both Coxeter graphs have the same number of vertices.

\begin{proposition}
Let~$\Gamma$ and~$\Omega$ be two Coxeter graphs of spherical type.
If~$A[\Gamma]$ and~$A[\Omega]$ are commensurable, then~$\Gamma$ and~$\Omega$ have the same number of vertices.
\end{proposition}

\begin{proof}
Suppose that~$A[\Gamma]$ and~$A[\Omega]$ are commensurable.
Let~$n$ be the number of vertices of~$\Gamma$ and let~$m$ be the number of vertices of~$\Omega$. We know that the cohomological dimension of~$A[\Gamma]$ is~$n$ and the cohomological dimension of~$A[\Omega]$ is~$m$ \citep[Proposition 3.1]{Paris1}. As every finite index subgroup of~$A[\Gamma]$ has the same cohomological dimension as~$A[\Gamma]$ and every finite index subgroup of~$A[\Omega]$ has the same cohomological dimension as~$A[\Omega]$, we have $n =m$.
\end{proof}

\noindent
In \myref{chapter4}{Section} we will prove the following result, which allows to reduce our problem to the study of two connected Coxeter graphs having the same number of vertices.

\begin{theorem}\label{reduction}
Let~$\Gamma$ and~$\Omega$ be two Coxeter graphs of spherical type. 
Let $\Gamma_1, \dots, \Gamma_p$ be the connected components of~$\Gamma$ and $\Omega_1, \dots, \Omega_q$ be the connected components of~$\Omega$.
Then~$A[\Gamma]$ and~$A[\Omega]$ are commensurable if and only if $p=q$ and~$A[\Gamma_i]$ and~$A[\Omega_i]$ are commensurable for every $i \in \{1, \dots, p\}$, up to permutation of the indices.
\end{theorem}

\noindent
Let~$G$ be a group.
A subgroup $H$ of~$G$ is a \emph{direct factor} of~$G$ is there is a subgroup~$K$ of~$G$ such that $G = H \times K$.
We say that~$G$ is \emph{indecomposable} if~$G$ does not have any non-trivial proper direct factor.
We say that~$G$ is \emph{strongly indecomposable} if~$G$ is infinite and every finite index subgroup~$H$ of~$G$ is indecomposable.
A \emph{strong Remak decomposition} of~$G$ is a finite index subgroup~$H$ of~$G$ with a direct product decomposition $H = H_1 \times \cdots \times H_p$ such that~$H_i$ is strongly indecomposable for every $i \in \{1, \dots, p\}$.
Two strong Remak decompositions of~$G$, $H = H_1 \times \cdots \times H_p$ and $H' = H_1' \times \cdots \times H_q'$, are said to be \emph{equivalent} if $p = q$ and~$H_i$ and~$H_i'$ are commensurable for every $i \in \{1, \dots, p\}$, up to permutation of the indices.

\bigskip\noindent
The center of a group~$G$ will be denoted by~$Z (G)$.
If~$\Gamma$ is a connected Coxeter graph of spherical type then, thanks to \citep{BrieskornSaito} and \citep{Deligne}, the center of~$A[\Gamma]$ is a cyclic infinite group. 
The quotient $A[\Gamma]/Z(A[\Gamma])$ will be denoted by~$\overline{A[\Gamma]}$ and it will play an important role in our study.
Moreover, we denote by $\theta : A [\Gamma] \to W [\Gamma]$ the canonical projection and by~$\CA [\Gamma]$ the kernel of~$\theta$.
As before, we let $\overline{\CA [\Gamma]} = \CA[\Gamma]/Z(\CA [\Gamma])$.
In \myref{chapter2}{Section}, we will prove that $Z(\CA [\Gamma]) \simeq \Z$ and $\CA [\Gamma] \simeq \overline{\CA[\Gamma]} \times Z (\CA [\Gamma])$ (see \autoref{tecnico}).
If~$\Gamma$ is reduced to a single vertex, then $\CA [\Gamma] = Z(\CA [\Gamma]) \simeq \Z$ and $\overline{\CA [\Gamma]} = \{1\}$.
Otherwise $\overline{\CA [\Gamma]} \neq \{1\}$.

\bigskip\noindent
The proof of \autoref{reduction} is based on the following result which will be proven in \myref{chapter3}{Section}.

\begin{theorem}\label{decomposition}
\hfill
\begin{itemize}
\item[(1)]
Let~$\Gamma$ be a connected Coxeter graph of spherical type which is not reduced to a single vertex.
Then~$\overline{\CA[\Gamma]}$ is strongly indecomposable. 
\item[(2)]
Let~$\Gamma$ be a Coxeter graph of spherical type and let $\Gamma_1, \dots, \Gamma_p$ be its connected components.
We suppose that each $\Gamma_1, \dots, \Gamma_k$ has at least two vertices and each of $\Gamma_{k+1}, \dots, \Gamma_p$ is reduced to a single vertex.
Then
\[
CA[\Gamma] = \overline{\CA[\Gamma_1]} \times \cdots \times \overline{\CA[\Gamma_k]} \times Z(\CA [\Gamma_1]) \times \cdots \times Z(\CA [\Gamma_p])
\]
is a strong Remak decomposition of~$A [\Gamma]$, and it is unique up to equivalence.
\end{itemize}
\end{theorem}

\noindent
A similar result for Coxeter groups is obtained in \citep{Paris5}. In order to finish the classification, we just need to compare the Artin groups associated to connected Coxeter graphs of spherical type with the same number of vertices. In \myref{chapter5}{Section} we prove the following result, which compares every group of this type with the corresponding Artin group of type~$A_n$.

\begin{theorem}\label{classification}
\hfill
\begin{itemize}
\item[(1)]
Let $n \ge 2$.
Then $A[A_n]$ and $A[B_n]$ are commensurable. 
\item[(2)]
Let $n \ge 4$.
Then $A[A_n]$ and $A[D_n]$ are not commensurable. 
\item[(3)]
Let $n \in \{6,7,8\}$.
Then $A[A_n]$ and $A[E_n]$ are not commensurable.
\item[(4)]
$A[A_4]$ and $A[F_4]$ are not commensurable.
\item[(5)]
Let $n \in \{3,4\}$.
Then $A[A_n]$ and $A[H_n]$ are not commensurable.
\item[(6)]
Let $p \ge 5$.
Then $A[A_2]$ and $A[I_2(p)]$ are commensurable.  
\end{itemize}
\end{theorem}

\noindent
The strategy of the proof of this theorem is the following. We use direct proofs to show Part~(1) and Part~(6). Using the fact that the abstract commensurator of $\overline{A[A_n]}$ is known to be a mapping class group of a punctured sphere (see \citealp{CharneyCrisp}), we show that, if $A[\Gamma]$ is commensurable with $A[A_n]$, then there is a homomorphism  $\varphi : \overline{A [\Gamma]} \to \SSS_{n+2} \times \{ \pm 1\}$ whose
kernel has no generalized torsion. Then, in Parts (2) to (5), in order to prove that $A[\Gamma]$ and $A[A_n]$ are not commensurable, we check in each case that the kernel of every homomorphism $\varphi : \overline{A [\Gamma]} \to \SSS_{n+2} \times \{ \pm 1\}$ has generalized torsion.

\medskip\noindent
The description of $\overline{A[D_4]}$ as the pure mapping class group of the three times punctured torus \citep{PL} as been recently used by \citet{Soroko} to apply the same techniques presented in this article to show that $A[D_4]$ is not commensurable with $A[F_4]$ and~$A[H_4]$. For the remaining cases, we have no hint on how to describe the abstract commensurator of one of the two groups, and this is needed in our argument. So, the following cases remain open:

\begin{itemize}
\item
For $n = 6,7,8$, we do not know if~$A[D_n]$ and~$A[E_n]$ are commensurable.  
\item
For $n=4$, we do not know if~$A[F_4]$ and~$A[H_4]$ are commensurable.
\end{itemize}

\section{A technical and useful result}\label{chapter2}

This section is devoted to some technical results (see \autoref{tecnico}) that will be the key to prove the main theorems of the forthcoming sections. These results are also interesting by themselves.   

\medskip\noindent
Let $\Gamma$ be a Coxeter graph of spherical type.
The \emph{Artin monoid} associated to~$\Gamma$ is the monoid~$A[\Gamma]^+$ having the same presentation as~$A[\Gamma]$, that is,
\[
A[\Gamma]^+ = \langle S \mid \Pi (s,t,m_{s,t}) = \Pi (t,s, m_{s,t}) \text{ for } s,t \in S,\ s \neq t,\ m_{s,t} \neq \infty \rangle^+\,.
\]
By \citep{BrieskornSaito} (see also \citep{Paris3}), $A[\Gamma]^+$ naturally injects in~$A[\Gamma]$.
We define a partial order~$\le_L$ on~$A[\Gamma]$ by $\alpha \le_L \beta$ if $\alpha^{-1} \beta \in A[\Gamma]^+$.
Also by \citep{BrieskornSaito}, the ordered set~$(A[\Gamma], \le_L)$ is a lattice.
We denote by~$\wedge_L$ and $\vee_L$ the lattice operations in $(A[\Gamma], \le_L)$.
In this case, the \emph{Garside element} of~$A[\Gamma]$ is defined as~$\Delta = \vee_LS$.
Again by \citep{BrieskornSaito} and \citep{Deligne} we know that, if~$\Gamma$ is connected, then the center of~$A[\Gamma]$ is infinite and cyclic, and it is generated by an element~$\delta$  of the form $\delta = \Delta^\kappa$, where $\kappa \in \{1,2\}$.
This element~$\delta$ will be called the  \emph{standard generator}  of~$Z(A[\Gamma])$.
We can also express~$\delta$ as follows.
Let $S=\{s_1, \dots, s_n\}$.
Then,  by \citep{BrieskornSaito}, $\delta = (s_1 s_2 \cdots s_n)^{\frac{h}{2}}$ if $\kappa=1$ and $\delta = (s_1 s_2 \cdots s_n)^{h}$ if $\kappa=2$, where $h$ is the Coxeter number of~$\Gamma$, that is, the order of $s_1 s_2 \cdots s_n$ in the associated Coxeter group.
These equalities do not depend on the choice when numbering the elements of~$S$.

\bigskip\noindent
Let~$\Gamma$ be a connected Coxeter graph of spherical type.
Let $z : A[\Gamma] \to \Z$ be the homomorphism such that $z(s)=1$ for every $s \in S$.
Hence, considering the later expression of~$\delta$, we have that $z (\delta) >0$.
The quotient $A[\Gamma]/Z(A[\Gamma])$ is denoted by $\overline{A[\Gamma]}$.
Moreover, recall that $\theta : A [\Gamma] \to W [\Gamma]$ is the canonical projection, $\CA [\Gamma]$ is the kernel of~$\theta$, and $\overline{\CA [\Gamma]} = \CA[\Gamma]/Z(\CA [\Gamma])$.

\begin{proposition}\label{tecnico}
Let $\Gamma, \Omega$ be two connected Coxeter graphs of spherical type. 
\begin{itemize}
\item[(1)]
If~$U$ is a finite index subgroup of~$A[\Gamma]$, then $Z (U) = Z(A [\Gamma]) \cap U$.
In particular, $Z (U)$ is an infinite cyclic group.
\item[(2)]
We have $\CA[\Gamma] \simeq \overline{\CA[\Gamma]} \times Z(\CA [\Gamma]) \simeq \overline{\CA [\Gamma]} \times \Z$.
\item[(3)]
$A [\Gamma]$ and $A[\Omega]$ are commensurable if and only if~$\overline{A [\Gamma]}$ and~$\overline{A[\Omega]}$ are commensurable.
\item[(4)]
The group $\overline{A[\Gamma]}$ injects in its commensurator $\Com (\overline{A [\Gamma]})$.
\end{itemize}
\end{proposition}

\begin{proof}
{\it Proof of Part (1).}
Let~$U$ be a finite index subgroup of~$A [\Gamma]$.
The inclusion $Z(A [\Gamma]) \cap U \subset Z(U)$ is obvious.
We need to show $Z(U) \subset Z(A [\Gamma]) \cap U$.
Let $\alpha \in Z(U)$ and $s \in S$.
As~$U$ is a finite index subgroup, there is $k \ge 1$ such that $s^k \in U$.
Then  $\alpha s^k \alpha^{-1} = s^k$ and, by  \citep[Corollary 5.3]{Paris2}, $\alpha s \alpha^{-1} = s$.
This proves that~$\alpha$ belongs to~$Z(A [\Gamma])$. 
To see that $Z(U)$ is infinite cyclic, notice that $Z(U)$ is a finite index subgroup of $Z(A [\Gamma])$, which is infinite cyclic because $\Gamma$ is connected.

\bigskip\noindent
{\it Proof of Part (2).}
Let $V = \oplus_{s \in S} \R e_s$ be a real vector space with a basis in one-to-one correspondence with~$S$.
By \citep{Bourbaki}, $W = W[\Gamma]$ has a faithful linear representation $\rho : W \to \GL (V)$ and~$\rho (W)$ is generated by reflections.
We denote by~$\HH$ the set of reflection hyperplanes of~$W$.
We let $V_\C = \C \otimes V$ and $H_\C = \C \otimes H$ for every $H \in \HH$.
Let also
\[
M = V_\C \setminus \left( \bigcup_{H \in \HH} H_\C \right)\,.
\]
Notice that~$M$ is a connected manifold of dimension $2\, |S|$.
By \citep{Brieskorn}, $\pi_1 (M) = \CA [\Gamma]$.

\bigskip\noindent
Let $h : V_\C\setminus \{ 0 \} \to \P V_\C$ be the Hopf fibration.
Let $\overline{M} = h(M)$ and denote by $h_\HH : M \to \overline{M}$ the restriction of~$h$ to~$M$.
Recall that the fiber of~$h_\HH$ is~$\C^*$.
As~$\HH$ is non-empty, we know that~$h_\HH$ is topologically a trivial fibration \citep[Proposition 5.1]{Orlik}.
In other words, $M$~is homeomorphic to $\overline{M} \times \C^*$, hence $\CA [\Gamma] = \pi_1(M) \simeq \pi_1 (\overline{M}) \times \Z$.
From this decomposition it follows that $Z(\CA [\Gamma]) \simeq Z(\pi_1 (\overline{M})) \times \Z$. 
But, thanks to Part~(1), $Z(\CA [\Gamma])$ is isomorphic to~$\Z$, which does not have any non-trivial direct product decomposition, hence $Z(\pi_1 (\overline{M})) =1$, $\pi_1 (\overline{M}) \simeq \CA [\Gamma]/Z(\CA [\Gamma]) = \overline{\CA [\Gamma]}$, and $\CA [\Gamma] \simeq \overline{\CA [\Gamma]} \times \Z = \overline{\CA [\Gamma]} \times Z(\CA [\Gamma])$.

\bigskip\noindent
{\it Proof of Part~(3).}
Suppose that~$A [\Gamma]$ and~$A [\Omega]$ are commensurable.
There is a finite index subgroup~$U$ of~$A [\Gamma]$ and a finite index subgroup~$V$ of~$A [\Omega]$ such that~$U$ is isomorphic to~$V$.
Let $\pi : A [\Gamma] \to \overline{A[\Gamma]}$ and $\pi': A [\Omega] \to \overline{A [\Omega]}$ be the corresponding canonical projections.
Then $\pi(U) = U/(Z(A [\Gamma]) \cap U)$ is a finite index subgroup of~$\overline{A [\Gamma]}$, $\pi' (V) = V/(Z(A [\Omega]) \cap V)$ is a finite index subgroup of~$\overline{A [\Omega]}$, and, by Part~(1), we have $\pi(U) = U/Z(U)$ and $\pi'(V) = V / Z(V)$. Hence $\pi(U)$ is isomorphic to~$\pi'(V)$. Therefore, $\overline{A [\Gamma]}$ and $\overline{A [\Omega]}$ are commensurable.

\bigskip\noindent
Suppose that~$\overline{A [\Gamma]}$ and~$\overline{A[\Omega]}$ are commensurable.
By Part~(1), $Z(\CA [\Gamma]) = \CA [\Gamma] \cap Z(A [\Gamma])$, then $\pi (\CA [\Gamma]) = \overline{\CA [\Gamma]}$ and~$\overline{\CA [\Gamma]}$ is a finite index subgroup of~$\overline{A [\Gamma]}$.
Likewise, $\overline{\CA [\Omega]}$ is a finite index subgroup of~$\overline{A [\Omega]}$, then~$\overline{\CA [\Gamma]}$ and~$\overline{\CA [\Omega]}$ are commensurable. This means that there are finite index subgroups~$\bar U$ of~$\overline{\CA [\Gamma]}$ and~$\bar V$ of~$\overline{\CA [\Omega]}$ such that~$\bar U$ and~$\bar V$ are isomorphic.
By Part~(2), $\CA [\Gamma] = \overline{\CA [\Gamma]} \times \Z$ and~$\CA [\Omega] = \overline{\CA[\Omega]} \times \Z$.
Let $U = \bar U \times \Z \subset \CA [\Gamma]$ and $V = \bar V \times \Z \subset \CA [\Omega]$.
Hence~$U$ is a finite index subgroup of~$\CA [\Gamma]$, $V$ is a finite index subgroup of~$\CA [\Omega]$, and~$U$ and~$V$ are isomorphic.
Thus, $\CA [\Gamma]$ and~$\CA [\Omega]$ are commensurable, so $A [\Gamma]$ and~$A [\Omega]$ are commensurable.

\bigskip\noindent
{\it Proof of Part~(4).}
If~$G$ is a group and $\alpha \in G$ we denote by $c_\alpha : G \to G$, $\beta \mapsto \alpha \beta \alpha^{-1}$, the conjugation by~$\alpha$.
Then we have a homomorphism $\iota_G : G \to \Com (G)$ sending $\alpha$ to the class of~$(G,G, c_\alpha)$.
Let $\iota = \iota_{\overline{A [\Gamma]}} : \overline{A [\Gamma]} \to \Com ( \overline{A [\Gamma]})$, and let $\alpha \in A [\Gamma]$ be such that $\pi (\alpha) \in \Ker(\iota)$, where $\pi : A [\Gamma] \to \overline{A [\Gamma]}$ is the corresponding canonical projection.
There is a finite index subgroup~$\overline{U}$ of~$\overline{A [\Gamma]}$ such that $\pi(\alpha)\, \pi(\beta)\, \pi (\alpha^{-1}) = \pi (\beta)$ for every $\beta \in \pi^{-1} (\bar U)$.
Let $s \in S$.
As~$\bar U$ has finite index in~$\overline{A [\Gamma]}$, there is  $k \ge 1$ such that $\pi(s^k) \in \bar U$.
We have $\pi(\alpha)\, \pi(s^k)\, \pi(\alpha^{-1}) = \pi(s^k)$, so $\pi(\alpha s^k \alpha^{-1} s^{-k})=1$ and then $\alpha s^k \alpha^{-1} s^{-k} \in \Ker(\pi) = Z(A [\Gamma]) = \langle \delta \rangle$. Hence there is  $\ell \in \Z$ such that $\alpha s^k \alpha^{-1} s^{-k} = \delta^\ell$.
Recall that $z : A[\Gamma] \to \Z$ is the homomorphism sending every element of $S$ to~$1$ and $z (\delta) >0$.
Then $0 = z(\alpha s^k \alpha^{-1} s^{-k}) = z(\delta^\ell) = \ell\, z(\delta)$ and $z(\delta)>0$, having that $\ell=0$ and $\alpha s^k \alpha^{-1} = s^k$.
By \citep[Corollary 5.3]{Paris1} it follows that $\alpha s \alpha^{-1} = s$.
This shows that $\alpha$ belongs to~$Z(A [\Gamma])$, so $\pi(\alpha)=1$ and $\iota$ is injective.
\end{proof}

\noindent
The proof of the following corollary is completely and explicitly included in the proof of the proposition above.

\begin{corollary}\label{corolariotecnico}
Let~$\Gamma$ be a connected Coxeter graph of spherical type.
Then $Z(\CA [\Gamma])$ is an infinite cyclic group. On the other hand, $\overline{\CA [\Gamma]}$ can be viewed as a subgroup of $\overline{A [\Gamma]}$, it has finite index in $\overline{A [\Gamma]}$, and its center is trivial. 
\end{corollary}

\section{Strong Remak decomposition}\label{chapter3}

In this section, we denote by~$\Gamma$ a Coxeter graph of spherical type associated to a Coxeter matrix~$M = (m_{s,t})_{s,t \in S}$.
Recall that our aim is to show \autoref{decomposition}.

\bigskip\noindent
Let~$G$ be a group and~$E$ be a subset of~$G$.
Recall that the \emph{normalizer} of~$E$ in~$G$ is $N_G(E) = \{ \alpha \in G \mid \alpha E \alpha^{-1} = E\}$ and the \emph{centralizer} of~$E$ in~$G$ is~$Z_G(E) = \{ \alpha \in G \mid \alpha e \alpha^{-1} = e \text{ for every } e \in E \}$.
If $E = \{e\}$, we just write $Z_G(e) = Z_G(\{e\})$ to refer to the centralizer of~$e$. We also recall that the center of~$G$ is denoted by~$Z(G)$.

\begin{lemma}\label{finiteindexofCA} Suppose that~$\Gamma$ is connected and different from a single vertex. 
Let~$U$ be a finite index subgroup of~$\overline{\CA [\Gamma]}$.
Then $Z (U) = Z_{\overline{\CA [\Gamma]}} (U) = \{ 1 \}$.
\end{lemma}

\begin{proof}
We just need to show that $Z_{\overline{\CA [\Gamma]}} (U) = \{ 1 \}$ because $Z(U) \subset Z_{\overline{\CA [\Gamma]}} (U)$.
This follows directly from the forth statement of \autoref{tecnico}. Indeed, as $U$ has finite index, any $\alpha$ element in $Z(U)$ is sent to the class $(G,G,1)$ via the homomorphism $\iota$ of the afore-mentioned proposition. Hence $\alpha $ sits in the kernel of $\iota$, which is trivial.
\end{proof}

\begin{lemma}\label{centralizers}
Let $G$ be a group, let $G_1, G_2$ be two subgroups of~$G$ such that $G = G_1 \times G_2$, and let~$H$ be a subgroup of~$G$.
Then $Z_G (H) = (Z_G(H) \cap G_1) \times (Z_G(H) \cap G_2)$.
\end{lemma}

\begin{proof}
The inclusion $(Z_G (H) \cap G_1) \times (Z_G (H) \cap G_2) \subset Z_G (H)$ is obvious.
Then we just need to show that $Z_G (H) \subset (Z_G (H) \cap G_1) \times (Z_G (H) \cap G_2)$.
Let $\alpha \in Z_G (H)$ and $\gamma \in H$.
We write $\alpha = (\alpha_1, \alpha_2)$ and $\gamma = (\gamma_1, \gamma_2)$ with $\alpha_1, \gamma_1 \in G_1$ and $\alpha_2, \gamma_2 \in G_2$.
We have $1 = \alpha \gamma \alpha^{-1} \gamma^{-1} = (\alpha_1 \gamma_1 \alpha_1^{-1} \gamma_1^{-1}, \alpha_2 \gamma_2 \alpha_2^{-1} \gamma_2^{-1})$, hence $\alpha_1 \gamma_1 \alpha_1^{-1} \gamma_1^{-1} = 1$.
Moreover, $\alpha_1 \gamma_2 \alpha_1^{-1} \gamma_2^{-1} = 1$, because $\alpha_1 \in G_1$ and $\gamma_2 \in G_2$, so $\alpha_1 \gamma \alpha_1^{-1} \gamma^{-1} = 1$.
Thus, $\alpha_1 \in (Z_G (H) \cap G_1)$.
Analogously, we can prove that $\alpha_2 \in (Z_G (H) \cap G_2)$.
\end{proof}

\bigskip\noindent
{\bf Proof of \autoref{decomposition}.}
We suppose that~$\Gamma$ is connected and different from a single vertex. Let~$U$ be a finite index subgroup of~$\overline{\CA [\Gamma]}$.
Let $U_1, U_2$ be two subgroups of~$U$ such that $U = U_1 \times U_2$ and let $\tilde{U}=U \times Z(\CA[\Gamma])$, which is included in $\overline{\CA [\Gamma]}\times Z(\CA[\Gamma])= \CA[\Gamma]$. Let $\tilde{U_1}=U_1$ and $\tilde{U_2}=U_2\times Z(\CA[\Gamma])$, having $\tilde{U}=\tilde{U_1}\times \tilde{U_2}.$ As~$\CA[\Gamma]$ has finite index in~$A[\Gamma]$, $\tilde{U}$ has finite index in~$A[\Gamma]$ and, by applying \citep[Theorem~5B]{Mar}, we know that either $\tilde{U_1}\subset Z(A[\Gamma])$ or $\tilde{U_2}\subset Z(A[\Gamma])$. 
Also by \autoref{tecnico}, we have that $ Z(A[\Gamma])\cap \tilde{U}=Z(\tilde{U})\subset Z(A[\Gamma]) \cap \CA[\Gamma] = Z(\CA [\Gamma])$. Then $\tilde U_1 \subset Z(\CA [\Gamma])$ or $\tilde U_2 \subset Z(\CA[\Gamma])$, so $U_1 = \{1\}$ or $U_2 = \{1\}$.
This shows the first part of the theorem. We still have to prove the second part.

\bigskip\noindent
Let $\Gamma_1, \dots, \Gamma_p$ be the connected components of~$\Gamma$.
We suppose that every $\Gamma_1, \dots, \Gamma_k$ has at least two vertices and that each of $\Gamma_{k+1}, \dots, \Gamma_p$ is reduced to a single vertex. 
We have that $\CA[\Gamma] = \CA[\Gamma_1] \times \cdots \times \CA [\Gamma_p]$.
By \autoref{tecnico}, $\CA [\Gamma_i] = \overline{\CA [\Gamma_i]} \times Z(\CA [\Gamma_i])$ for every $i \in \{1, \dots, k\}$ and $\CA[\Gamma_i] = Z (\CA [\Gamma_i])$ for every $i \in \{k+1, \dots, p\}$, hence
\begin{equation}\label{eq3_1}
\CA [\Gamma] = \overline{\CA [\Gamma_1]} \times \cdots \times \overline{\CA[\Gamma_k]} \times Z(\CA [\Gamma_1]) \times \cdots \times Z(\CA [\Gamma_p])\,.
\end{equation}
Then, $\overline{\CA [\Gamma_i]}$ is strongly indecomposable for every $i \in \{1, \dots, k\}$.
Moreover, $Z(\CA[\Gamma_i])$ is strongly indecomposable, because $Z (\CA [\Gamma_i]) \simeq \Z$, for every $i \in \{1, \dots, p\}$.
Therefore, (\ref{eq3_1}) is a strong Remak decomposition of~$A [\Gamma]$.
Now, we take a strong Remak decomposition of~$A [\Gamma]$ of the form $H = H_1 \times \cdots \times H_m$ and turn to prove that it is equivalent to~(\ref{eq3_1}).

\medskip

\begin{claim}\label{claim1}
We can assume that $H \subset \CA [\Gamma]$.
\end{claim}

\noindent
{\it Proof of \autoref{claim1}.}
Let $H_i' = H_i \cap \CA [\Gamma]$ for every $i \in \{1, \dots, m\}$ and $H'= H_1' \times \cdots \times H_m'$.
Since~$\CA [\Gamma]$ is a finite index subgroup of~$A[\Gamma]$, $H_i'$ has finite index in~$H_i$ for every $ i \in \{1, \dots, m\}$. This means that $H'$ has finite index in~$H$ and therefore $H'$ has finite index in~$A [\Gamma]$.
As~$H_i$ is strongly indecomposable and~$H_i'$ has finite index in~$H_i$, $H_i'$ is strongly indecomposable, for every $i \in \{1, \dots, m\}$.
Then $H' = H_1' \times \cdots \times H_m'$ is a strong Remak decomposition of~$A [\Gamma]$.
By construction, this decomposition is equivalent to $H = H_1 \times \cdots \times H_m$ and $H'$ is included in~$\CA[\Gamma]$.
This finishes the proof of \autoref{claim1}.

\medskip\noindent
Let $\tilde B = Z(\CA [\Gamma_1]) \times \cdots \times Z  (\CA[\Gamma_p]) \simeq \Z^p$ and $B = H \cap \tilde B$. Set $K_i = H \cap \overline{\CA [\Gamma_i]}$ for every $i \in \{1, \dots, k\}$. As~$H$ has finite index in~$\CA [\Gamma]$, $K_i$ has finite index in~$\overline{\CA [\Gamma_i]}$ for every $i \in \{1, \dots, k\}$ and $B$ has finite index in~$\tilde B$.

\medskip

\begin{claim}\label{claim3}
We have that $Z (H) = B$.
\end{claim}

\noindent
{\it Proof of \autoref{claim3}.}
Let $\alpha\in Z(H)\subset \CA [\Gamma]$. Then, by \autoref{centralizers}, $\alpha$ can be expressed as $\alpha = \alpha_1 \cdots \alpha_k\beta$, where $\alpha_i \in \overline{\CA [\Gamma_i]}\cap Z_{{\CA [\Gamma]}}(H)$ for every $i \in \{1, \dots,k\}$ and $\beta \in \tilde B$. Since $K_i\subset H$, $\alpha_i$ commutes with every element in $K_i$, so $\alpha_i\in Z_{\overline{\CA [\Gamma_i]}}(K_i)$. By \autoref{finiteindexofCA}, $Z_{\overline{\CA [\Gamma_i]}}(K_i)=\{1\}$, hence $\alpha_i=1$.  Therefore, $\alpha = \beta \in \tilde B \cap H = B$. This proves that $Z(H)\subset B$. To see $B\subset Z(H)$, just notice that $B=Z(\CA[\Gamma])\cap H\subset Z(H)$, because $H\subset\CA[\Gamma]$ by \autoref{claim1}.
This finishes the proof of \autoref{claim3}.

\bigskip\noindent
Let $\hat K_i= K_1\times \cdots \times K_{i-1}\times K_{i+1}\times \cdots \times K_k \times B$ and $L_i=(\overline{\CA [\Gamma_i]}\times \tilde B)\cap H$, for every $i \in \{1, \dots, k\}$.

\medskip

\begin{claim}\label{claim2}
Let $i \in \{1, \dots, k\}$.
Then $Z_H (\hat K_i) = L_i$ and $L_i=(L_i\cap H_1)\times \cdots \times (L_i\cap H_m)$.
\end{claim}

\medskip\noindent
{\it Proof of \autoref{claim2}.}
By \autoref{centralizers}, we have that $$Z_{\CA [\Gamma]} (\hat K_i) = (Z_{\CA [\Gamma]}  (\hat K_i) \cap \overline{\CA [\Gamma_1]}) \times \cdots \times (Z_{\CA [\Gamma]} (\hat K_i) \cap \overline{\CA [\Gamma_k]}) \times (Z_{\CA [\Gamma]} (\hat K_i) \cap \tilde B).$$
Let $j \in \{1, \dots, k\}$ be such that $j \neq i$.
Then, by \autoref{finiteindexofCA}, $(Z_{\CA [\Gamma]} (\hat K_i) \cap \overline{\CA [\Gamma_j]}) \subset Z_{\overline{\CA [\Gamma_j]}} (K_j) = \{1\}$.
Moreover, $(Z_{\CA [\Gamma]} (\hat K_i) \cap \overline{\CA [\Gamma_i]}) = \overline{\CA [\Gamma_i]}$ and $(Z_{\CA [\Gamma]} (\hat K_i) \cap \tilde B) = \tilde B$.
Therefore $Z_{\CA [\Gamma]} (\hat K_i) = \overline{\CA [\Gamma_i]} \times \tilde B$ and $Z_H (\hat K_i) = Z_{\CA [\Gamma]} (\hat K_i) \cap H = L_i$.
Finally, by \autoref{centralizers}, $L_i= Z_H (\hat K_i) =(L_i\cap H_1)\times \cdots \times (L_i\cap H_m)$.
This finishes the proof of \autoref{claim2}.

\medskip

\begin{claim}\label{claim4}
Let $i \in \{1, \dots, k\}$. Then $Z(L_i)=B$ and $L_i/B$ is strongly indecomposable. Also, there is $\chi (i) \in \{1, \dots, m\}$ such that $L_i/B = (L_i \cap H_{\chi (i)})/Z(H_{\chi (i)})$ and $L_i \cap H_j = Z(H_j)$ if $j \neq \chi (i)$.
\end{claim}

\noindent
{\it Proof of \autoref{claim4}.}
Since $B\subset Z(\CA[\Gamma])$ and $B\subset L_i \subset H$, we have $B\subset Z(L_i)$. Now, take $\alpha\in Z(L_i)$. As $L_i$ is a subgroup of $\overline{\CA [\Gamma_i]} \times \tilde B$, by \autoref{centralizers} we can write $\alpha$ in the form $\alpha= \alpha_i \beta$, where $\alpha_i\in \overline{\CA [\Gamma_i]} \cap Z_{CA[\Gamma]}(L_i)$ and $\beta\in \tilde B$. Since $K_i\subset L_i$, $\alpha_i$ commutes with every element in $K_i$, so $\alpha_i\in Z_{\overline{\CA [\Gamma_i]}}(K_i)$. By \autoref{finiteindexofCA}, $Z_{\overline{\CA [\Gamma_i]}}(K_i)=\{1\}$, hence $\alpha_i=1$.  Therefore, $\alpha = \beta \in \tilde B \cap H = B$ and then $Z(L_i)\subset B$.

\medskip\noindent
Let $\pi: \overline{\CA [\Gamma_i]} \times \tilde B \rightarrow \overline{\CA [\Gamma_i]}$ be the projection homomorphism and let $\pi'$ be the restriction of $\pi$ to $L_i$. Then $\Ker(\pi')=\Ker(\pi)\cap L_i=\tilde B \cap L_i \subset Z(L_i)=B$. On the other hand, $B\subset \tilde B \cap L_i$, hence $\Ker(\pi')=B$. Using the first isomorphism theorem, we have that $L_i/ B \simeq \pi(L_i)$. As $L_i$ has finite index in $ \overline{\CA [\Gamma_i]}\times \tilde B$, $\pi(L_i)$ has finite index in $\overline{\CA [\Gamma_i]}$, which is strongly indecomposable. This implies that $L_i/B$ is strongly indecomposable.

\medskip\noindent
By \autoref{claim2}, we have that $L_i=(L_i\cap H_1)\times \cdots \times (L_i\cap H_m)$.
If we quotient this equality by $B = Z(H) = Z(H_1) \times \cdots \times Z(H_m)$, we get $L_i/B = (L_i \cap H_1)/Z(H_1) \times \cdots \times (L_i \cap H_m)/Z(H_m)$.
We already know that~$L_i/B$ is strongly indecomposable, so there is $\chi (i) \in \{1, \dots, m\}$ such that $L_i/B = (L_i \cap H_{\chi (i)})/Z(H_{\chi (i)})$ and $(L_i \cap H_j)/ Z(H_j) = \{1\}$ if $j \neq \chi (i)$. This implies that $L_i \cap H_j \subset Z(H_j)$ if $j \neq \chi (i)$. Finally, notice that $Z(H_i)\subset B \subset L_i$.
This proves \autoref{claim4}.

\bigskip\noindent
For $j \in \{1, \dots, m\}$ we denote  by~$f_j : H \to H_j$ the projection of~$H$ on~$H_j$.
Let $K=K_1\times \cdots \times K_k\times B$. For $i \in \{1, \dots, k\}$ we denote by $g_i : K \to K_i$ the projection of~$K$ on~$K_i$, and we denote by $h : K \to B$ the projection of~$K$ on~$B$.
Notice that, since $K_i\times B$ is a subgroup of $L_i$, $K_i$ is a subgroup of $L_i/B$. Then, $K_i$ injects into $L_i$ and into $L_i/B$, which by \autoref{claim4} is isomorphic to $(L_i \cap H_{\chi (i)})/Z(H_{\chi (i)})$.  This means that the composition
$$K_i\hookrightarrow L_i \xrightarrow[]{f_{\chi(i)}} L_i\cap H_{\chi(i)} $$
has to be injective. In other words, the restriction $f_{\chi (i)} |_{K_i} : K_i \to H_{\chi (i)}$ is injective. Also by \autoref{claim4}, if $j \neq \chi (i)$, 
$$K_i\hookrightarrow L_i \xrightarrow[]{f_{j}} L_i\cap H_{j}=Z(H_j) $$
Hence, we have $f_j (K_i) \subset Z(H_j)$.

\medskip\noindent
Let $\psi_i : K_i \to B$ be the map defined by $\psi_i (\alpha) = \prod_{j \neq \chi (i)} f_j (\alpha)^{-1}$.
As~$B$ is an abelian group, $\psi_i$~is a well-defined homomorphism.
Let $\psi : K \to B$ be the map defined by $\psi (\alpha) = \prod_{i=1}^k (\psi_i \circ g_i) (\alpha)$.
Then, again, $\psi$ is a well-defined homomorphism because~$B$ is abelian.
Also, notice that $\psi (\beta) = 1$ for every $\beta \in B$.
If $\varphi : K \to K$ is the map defined by $\varphi (\alpha) = \alpha\, \psi(\alpha)$, 
it is clear that~$\varphi$ is a homomorphism.
In addition, as $\psi (\beta) = 1$ for every $\beta \in B$, $\varphi$ is invertible and $\varphi^{-1}$ is defined by $\varphi^{-1} (\alpha) = \alpha \, \psi(\alpha)^{-1}$.

\medskip

\begin{claim}\label{claim5}
For every $i \in \{1, \dots, k\}$ we have $\varphi (K_i) \subset H_{\chi (i)}$.
\end{claim}

\noindent
{\it Proof of \autoref{claim5}.}
Let $i \in \{1, \dots, k\}$ and $\alpha \in K_i$.
For $\ell \in \{1, \dots, k\}$, $ \ell \neq i$, we have that $g_\ell(\alpha) = 1$, then $(\psi_\ell \circ g_\ell)(\alpha) = 1$ and $\psi (\alpha) = \psi_i (\alpha)$.
Moreover, $\alpha = \prod_{j=1}^m f_j (\alpha)$, having $\varphi (\alpha) = f_{\chi (i)} (\alpha) \in H_{\chi (i)}$.
This finishes the proof of \autoref{claim5}.

\bigskip\noindent
Up to applying~$\varphi$ we can assume that $K_i \subset H_{\chi (i)}$ for every $i \in \{1, \dots, k\}$.

\medskip

\begin{claim}\label{claim6} \hfill
\begin{itemize}
\item[(1)]
For every $i \in \{1, \dots, k\}$ we have $f_{\chi(i)} (K) = K_i$ and $f_{\chi (i)} (\hat K_i) = \{1 \}$. Moreover, $K_i$ is a finite index subgroup of~$H_{\chi (i)}$.
\item[(2)]
For every $i,j \in \{1, \dots, k\}$, $i \neq j$, we have $\chi(i) \neq \chi (j)$.
\end{itemize}
\end{claim}

\noindent
{\it Proof of \autoref{claim6}.}
As~$K$ has finite index in~$H$, $f_{\chi (i)}(K)$ has finite index in~$H_{\chi (i)}$.
Notice also that $f_{\chi (i)} (K_i) = K_i$ because $K_i \subset H_{\chi (i)}$. 
We have that $K = K_i \times \hat K_i$, so $f_{\chi (i)} (K) = f_{\chi(i)} (K_i)\, f_{\chi (i)} (\hat K_i) = K_i\, f_{\chi (i)} (\hat K_i)$ and $[K_i, f_{\chi (i)} (\hat K_i)] = \{1 \}$.
Moreover, by \autoref{finiteindexofCA}, $K_i \cap f_{\chi (i)} (\hat K_i) \subset Z(K_i) = \{1\}$.
Hence, $f_{\chi (i)} (K) \simeq K_i \times f_{\chi (i)} (\hat K_i)$.
As $H_{\chi (i)}$ is strongly indecomposable, we have that $f_{\chi (i)} (\hat K_i) = \{1\}$ and $f_{\chi (i)} (K) = K_i$.
On the other hand, let $j \in \{1, \dots, k\}$ be such that $j \neq i$.
As~$K_j$ is a subgroup of~$\hat K_i$, we have that $f_{\chi(i)} (K_j) = \{1\} \neq K_j$, and then $\chi(i) \neq \chi(j)$.
This finishes the proof of \autoref{claim6}.

\bigskip\noindent
By the results from above, $m \ge k$ and we can suppose that $\chi (i) = i$ for every $i \in \{1, \dots, k\}$, up to renumbering the $H_i$'s.
Recall that $B = Z(H) = Z(H_1) \times \cdots \times Z(H_m)$ and $Z(H_i) = f_i (B)$ for every $i \in \{1, \dots, m\}$.
If $i \in \{1, \dots, k\}$, then $B \subset \hat K_i$, so by \autoref{claim6}, $\{1\} = f_i (B) = Z(H_i)$.
Hence, $B = Z(H_{k+1}) \times \cdots \times Z(H_m)$.
We also have that $K = K_1 \times \cdots \times K_k \times B$ is a subgroup of $K_1 \times \cdots \times K_k \times H_{k+1} \times \cdots \times H_m$, and that~$K$ is a finite index subgroup of $H = H_1 \times \cdots \times H_k \times H_{k+1} \times \cdots \times H_m$. Therefore~$B$ has finite index in $H_{k+1} \times \cdots \times H_m$.

\bigskip\noindent
For $j \in \{ k+1, \dots, m\}$, we let $B_j = B \cap H_j$.
As~$B$ is a finite index subgroup of $H_{k+1} \times \cdots \times H_m$, $B_j$~has finite index in~$H_j$.
In addition, as~$H_j$ is strongly indecomposable, $B_j$ is indecomposable.
The group~$B_j$ is a subgroup of~$B \simeq \Z^p$ and it is indecomposable, so $B_j \simeq \Z$. 
Let $B' = B_{k+1} \times \cdots \times B_m$.
Then $B'$ is a finite index subgroup of $H_{k+1} \times \cdots \times H_m$ and $B'$ has finite index in~$B$.
As $B' \simeq \Z^{m-k}$ and $B \simeq \Z^p$, it follows that $m-k = p$.

\bigskip\noindent
For $i \in \{1, \dots, k\}$, $K_i$ is a finite index subgroup of both~$H_i$ and $\overline{\CA [\Gamma_i]}$, so~$H_i$ and~$\overline{\CA [\Gamma_i]}$ are commensurable.
Moreover, for every $i \in \{ k+1, \dots, m\}$, $Z (\CA(\Gamma_{i-k})) \simeq \Z \simeq B_i$, and~$B_i$ is a finite index subgroup of~$H_i$. Therefore $Z (\CA [\Gamma_{i-k}])$ and $H_i$ are commensurable. \qed

\section{Reduction to the connected case}\label{chapter4}

\autoref{reduction} is a consequence of \autoref{decomposition} and the following proposition.

\begin{proposition}\label{prop41}
Let $G$ and $G'$ be two infinite groups.
We suppose that~$G$ (resp. $G'$) has a unique strong Remak decomposition up to equivalence, $H = H_1 \times \cdots \times H_p$ (resp. $H' = H_1' \times \cdots \times H_q'$). 
Then~$G$ is commensurable with~$G'$ if and only if $p=q$ and, up to permutation of the factors, $H_i$ is commensurable with~$H_i'$ for every $i \in \{1, \dots, p \}$.
\end{proposition}

\begin{proof}

Suppose that $G$ and $G'$ are commensurable.
There is a finite index subgroup~$K$ of~$G$ and a finite index subgroup~$K'$ of~$G'$ such that $K \simeq K'$.
Let $\varphi : K \to K'$ be an isomorphism between~$K$ and~$K'$.
For every $i \in \{1, \dots, p\}$ we take $K_i = K \cap H_i$ and $U = K_1 \times \cdots \times K_p$.
As~$K$ has finite index in~$G$, $K_i$ has finite index in~$H_i$.
It follows that~$U$ is a finite index subgroup of~$H$ (and of~$G$) and $U = K_1 \times \cdots \times K_p$ is a strong Remak decomposition of~$G$.
The group~$U$ is a finite index subgroup of~$K$, so $\varphi(U) = \varphi(K_1) \times \cdots \times \varphi (K_p)$ is a finite index subgroup of $\varphi(K) = K'$ and then also a finite index subgroup of~$G'$.
The subgroups $\varphi (K_i)$ ($i \in \{1, \dots, p\})$ are strongly indecomposable, hence $\varphi(U) = \varphi(K_1) \times \cdots \times \varphi (K_p)$ is a strong Remak decomposition of~$G'$.
As~$G'$ has only one decomposition of that form (up to equivalence), we have that $p=q$ and~$\varphi (K_i)$ is commensurable with~$H_i'$ for every $i \in \{1, \dots, p\}$, up to permutation of the factors.
Also, as $K_i \simeq \varphi(K_i)$ is a finite index subgroup of~$H_i$, it follows that $H_i$ and $H_i'$ are commensurable for every $i \in \{1, \dots, p\}$.

\bigskip\noindent
Suppose that $p = q$ and that $H_i$ is commensurable with~$H_i'$ for every $i \in \{1, \dots, p\}$.
There is a finite index subgroup~$K_i$ of~$H_i$ and a finite index subgroup~$K_i'$ of~$H_i'$ such that $K_i \simeq K_i'$.
Take $U = K_1 \times \cdots \times K_p$ and $U' = K_1' \times \cdots \times K_p'$.
As $K_i$ has finite index in~$H_i$ for every $i$, the subgroup~$U$ has finite index in~$H$ and also has finite index in~$G$.
Analogously, $U'$ has finite index in~$G'$.
It is obvious that~$U$ and~$U'$ are isomorphic. Therefore, $G$ and $G'$ are commensurable.
\end{proof}

\bigskip\noindent
{\bf Proof of \autoref{reduction}.}
Let $\Gamma$ and~$\Omega$ be two Coxeter graphs of spherical type.
Let $\Gamma_1, \dots, \Gamma_p$ be the connected components of~$\Gamma$ and $\Omega_1, \dots, \Omega_q$ be the connected components of~$\Omega$.
If $p=q$ and~$A[\Gamma_i]$ and~$A[\Omega_i]$ are commensurable for every $i \in \{1, \dots, p\}$, then it is clear that~$A [\Gamma]$ and~$A [\Omega]$ are commensurable.
Then suppose that~$A [\Gamma]$ and~$A [\Omega]$ are commensurable. We need to show that $p=q$ and that~$A [\Gamma_i]$ and~$A [\Omega_i]$ are commensurable for every $i \in \{1, \dots, p\}$ up to permutation of the indices.

\bigskip\noindent
Suppose that every $\Gamma_1, \dots, \Gamma_k$ has a least two vertices and that each of $\Gamma_{k+1}, \dots, \Gamma_p$ is reduced to a single vertex.
Analogously, suppose that every $\Omega_1 , \dots, \Omega_\ell$ has a least two vertices and that each of $\Omega_{\ell+1}, \dots, \Omega_q$ is reduced to a single vertex.
By \autoref{decomposition}, $\CA [\Gamma] = \overline{\CA [\Gamma_1]} \times \cdots \times \overline{\CA [\Gamma_k]} \times Z(\CA [\Gamma_1]) \times \cdots \times Z (\CA [\Gamma_p])$ and $\CA [\Omega] = \overline{\CA [\Omega_1]} \times \cdots \times \overline{\CA [\Omega_\ell]} \times Z(\CA [\Omega_1]) \times \cdots \times Z(\CA [\Omega_q])$ are strong Remak decompositions of~$A[\Gamma]$ and~$A[\Omega]$, respectively, and they are unique up to equivalence.
Let $i \in \{1, \dots, k\}$ and $j \in \{1, \dots, q\}$.
Let $U$ be a finite index subgroup of~$\overline{\CA [\Gamma_i]}$ and let~$V$ be a finite index subgroup of~$Z (\CA [\Omega_j])$.
By \autoref{finiteindexofCA} we have that $Z (U) = \{ 1 \}$.
Moreover, $V$~is a finite index subgroup of $Z (\CA [\Omega_j]) \simeq \Z$, hence $V \simeq \Z$.
Then, $U$ and $V$ are not isomorphic.
This shows $\overline{\CA [\Gamma_i]}$ and $Z (\CA [\Omega_j])$ are not commensurable.

\bigskip\noindent
By applying \autoref{prop41}, we know that $k = \ell$, $p=q$ and that $\overline{ \CA [\Gamma_i]}$ and $\overline{\CA [\Omega_i]}$ are commensurable for every $i \in \{1, \dots, k\}$, up to permutation of the indices.
Let $i \in \{1, \dots, k\}$.
As $\overline{ \CA [\Gamma_i]}$ and $\overline{\CA [\Omega_i]}$ are commensurable, by \autoref{corolariotecnico}, $\overline {A [\Gamma_i]}$ and~$\overline{ A[\Omega_i]}$ are commensurable. Then, by \autoref{tecnico}, $A [\Gamma_i]$ and~$A [\Omega_i]$ are commensurable.
Let $i \in \{k+1, \dots, p\}$.
Thus, $A [\Gamma_i] \simeq \Z \simeq A[\Omega_i]$, having that~$A [\Gamma_i]$ and~$A [\Omega_i]$ are commensurable.
\qed

\medskip

\begin{remark}
An alternative proof of \autoref{reduction}, based on \citep[Theorem~B]{KKL}, has been communicated to us by one of the referees. The idea is as follows. Consider the decomposition (\ref{eq3_1}) given in the proof of \autoref{decomposition}. We can show using \citep{CalvezWiest}, \citep{Sisto} and \citep[Proposition~4.2]{KKL} that each~$\overline{\CA[\Gamma_i]}$ is of coarse type I. It follows by \citep[Theorem~B]{KKL} that the decomposition~(\ref{eq3_1}) is unique up to quasi-isometry. To pass from quasi-isometry to commensurability we have to apply again \citep[Theorem~B]{KKL} in the following manner. Let~$\Gamma$ and~$\Omega$ be two Coxeter graphs of spherical type. We assume that~$\CA[\Gamma]$ and~$\CA[\Omega]$ are commensurable and we consider the same decompositions of~$\CA[\Gamma]$ and~$\CA[\Omega]$ as in the above proof. Then, there exist finite index subgroups~$U$ of~$\CA[\Gamma]$ and~$V$ of~$\CA[\Omega]$ and an isomorphism $f\colon U \rightarrow V$. We extend $f$ to a quasi-isometry $f\colon \CA[\Gamma]\rightarrow \CA[\Omega]$. By applying \citep[Theorem~B]{KKL}, we obtain that $p=q$, $k=\ell$ and, for each $i\in \{i,\dots,k\}$ there exists $j\in\{1,\dots,k\}$ such that the composition $\overline{\CA[\Gamma_i]}\rightarrow \CA[\Gamma]\rightarrow \CA[\Omega]\rightarrow \overline{\CA[\Omega_j]}$ is a quasi-isometry, where the first map is the inclusion, the second map is $f$, and the third map is the projection. This map restricted to $U\cap \overline{\CA[\Gamma_i]}$ is an injective homomorphism whose image must be of finite index in $\CA[\Omega_j]$.

\end{remark}

\section{Comparison with the Artin group of type~$A_{\lowercase{n}}$}\label{chapter5}

Let $n \in \N$, $n \ge 2$, and let~$\Gamma$ be a connected Coxeter graph of spherical type with $n$~vertices.
Recall that the aim of this section is to determine whether~$A[\Gamma]$ and~$A[A_n]$ are commensurable or not.
We start with the cases where~$A[\Gamma]$ and~$A[A_n]$ are commensurable.

\begin{lemma}\label{AnBn}
Let $n \ge 2$.
Then~$A[A_n]$ and~$A[B_n]$ are commensurable.
\end{lemma}

\begin{proof}
Let $\theta: A[A_n] \to W[A_n]$ be the quotient homomorphism and let~$H$ be the subgroup of~$W[A_n]$ generated by $\{s_2, \dots, s_n\}$.
By \citep{Lambr1}, $\theta^{-1}(H)$ is isomorphic to~$A[B_n]$. It has finite index in~$A[A_n]$ because~$W[A_n]$ is finite, hence~$A[A_n]$ and~$A[B_n]$ are commensurable. 
\end{proof}

\begin{lemma}
Let $p \ge 5$.
Then $A[A_2]$ and $A[I_2(p)]$ are commensurable.
\end{lemma}

\begin{proof}
Let $\Gamma = I_2 (p)$.
Then $A [\Gamma] = \langle s,t \mid \Pi (s,t,p) = \Pi(t,s,p) \rangle$.
We consider the construction of the proof of \autoref{tecnico}\,(2).
Let $V = \R e_s \oplus \R e_t$.
By \citep{Bourbaki}, $W = W[\Gamma]$ has a faithful linear representation $\rho : W \to \GL (V)$, and~$\rho (W)$ is generated by reflections.
In our case, $W$~is the dihedral group of order~$2p$ and $\rho : W \to \GL (V)$ is the standard representation of~$W$.
Let~$\HH$ be the set of reflection lines of~$W$. 
Take $V_\C = \C \otimes V$, $H_\C = \C \otimes H$ for every $H \in \HH$, and $M = V_\C \setminus \left( \cup_{H \in \HH} H_\C \right)$.
Let $h : V_\C \setminus \{ 0\} \to \P V_\C$ be the Hopf fibration and $\overline{M} = h(M)$.
Thanks to the proof of \autoref{tecnico}\,(2), we know that $\pi_1 (\overline{M}) = \overline{\CA [\Gamma]}$.

\bigskip\noindent
In this case, $\P V_\C$ is the complex projective line and~$\overline{M}$ is the complement of~$|\HH| = p$ points in~$\P V_\C$, hence $\overline{\CA [\Gamma]} = \pi_1 (\overline{M})$ is isomorphic to the free group~$F_{p-1}$ of rank~$p-1$.
Analogously, $\overline{\CA [A_2]}$ is isomorphic to~$F_2$.
As~$F_{p-1}$ is isomorphic to a finite index subgroup of~$F_2$, it follows that~$\overline {\CA [\Gamma]}$ and~$\overline{\CA [A_2]}$ are commensurable.
By \autoref{corolariotecnico}, we have that~$\overline{A [A_2]}$ and~$\overline{A [\Gamma]}$ are commensurable. Therefore, by \autoref{tecnico}, $A[A_2]$ and $A [\Gamma]$ are commensurable.
\end{proof}

\bigskip\noindent
Let $\Sigma = \Sigma_{g,b}$ be the orientable surface of genus~$g$ with $b$~boundary components.
Let~$\PP_n$ be a collection of $n$~different points in the interior of~$\Sigma$.
Recall that the \emph{mapping class group} of the pair $(\Sigma, \PP_n)$, denoted by $\MM (\Sigma, \PP_n)$, is the group of isotopy classes of homeomorphisms $h : \Sigma \to \Sigma$ that preserve the orientation, fix the boundary of~$\Sigma$ pointwise and preserve~$\PP_n$ setwise. 
The \emph{extended mapping class group} of the pair $(\Sigma, \PP_n)$, denoted by~$\MM^* (\Sigma, \PP_n)$, is the group of isotopy classes of homeomorphisms  $h : \Sigma \to \Sigma$ that fix the boundary of~$\Sigma$ pointwise and preserve $\PP_n$ setwise.
Notice that, if the surface $\Sigma$ has non-empty boundary, the homeomorphisms fixing this boundary pointwise cannot change the orientation of~$\Sigma$ and we have $\MM^*(\Sigma, \PP_n)=\MM(\Sigma, \PP_n)$. 
Otherwise, $\MM (\Sigma, \PP_n)$ has index~$2$ in~$\MM^* (\Sigma_n,\PP_n)$.

\bigskip\noindent
Denote by $\SSS_n$ the permutation group of $\{1,\dots,n\}$. The action of $\MM^*(\Sigma, \PP_n)$ on~$\PP_n$ induces a homomorphism $\theta' : \MM^* (\Sigma, \PP_n) \to \SSS_n$, whose kernel is denoted by $\PP\MM^*(\Sigma, \PP_n)$.
On the other hand, we can define another homomorphism $\omega : \MM^* (\Sigma, \PP_n) \to \{ \pm 1 \}$ sending an element $h \in \MM^* (\Sigma, \PP_n)$ to $1$ if it preserves the orientation and to $-1$ otherwise. 
Notice that the kernel of~$\omega$ is $\MM(\Sigma, \PP_n)$. 
These two homomorphisms lead to the construction of the homomorphism $\hat{\theta}: \MM^* (\Sigma, \PP_n) \to \SSS_n \times \{ \pm 1\}$ defined by $h\mapsto (\theta'(h), \omega(h))$.
The kernel of~$\hat \theta$ is called the \emph{pure mapping class group} of the pair $(\Sigma, \PP_n)$ and it is denoted by $\PP \MM (\Sigma, \PP_n)$.

\bigskip\noindent
These mapping class groups and the problem that we are studying are related by the following theorem.

\begin{theorem}[\citealp{CharneyCrisp}]\label{CharneyCrisp}
Let $\Sigma=\Sigma_{0,0}$ and let $\PP_{n+2}$ be a family of $n+2$~points in~$\Sigma$.
Then $\Com( \overline{ A [A_n]}) \simeq \MM^* (\Sigma, \PP_{n+2})$.
\end{theorem}

\begin{lemma}\label{PMCA}
Let $\Sigma=\Sigma_{0,0}$ and let~$\PP_{n+2}$ be a family of $n+2$~points in~$\Sigma$.
Then $\Ker(\hat{\theta}) = \PP\MM(\Sigma, \PP_{n+2}) \simeq \overline{\CA[A_{n}]}$.
\end{lemma}

\begin{proof}
Let~$\BB_{n+1}$ be the braid group on $n+1$~strands.
By~\citep{Artin1}, $\MM (\Sigma_{0,1}, \PP_{n+1}) = \BB_{n+1} = A[A_{n}]$ and $\PP\MM (\Sigma_{0,1}, \PP_{n+1}) = \CA [A_{n}]$.
Let~$\delta$ be the standard generator of~$Z(A[A_{n}])$.
It is well-known that $\delta \in \CA [A_n]$ and $Z (A[A_n]) = Z(\CA[A_n]) = \langle \delta \rangle$.
Notice that $\delta$, seen as an element of $\PP\MM (\Sigma_{0,1}, \PP_{n+1})$, is the Dehn twist about the boundary component of~$\Sigma_{0,1}$. Then $\overline{\CA [A_{n}]} = \PP \MM (\Sigma_{0,1}, \PP_{n+1}) / \langle \delta \rangle = \PP\MM (\Sigma_{0,0}, \PP_{n+2})$; see, for example \citep{ParRol1}.
\end{proof}

\bigskip\noindent
Let~$G$ be a group.
We say that an element $\alpha \in G$ is a \emph{generalized torsion} element if there are $p \ge 1$ and $\beta_1, \dots, \beta_p \in G$ such that $(\beta_1\alpha\beta_1^{-1}) (\beta_2 \alpha \beta_2^{-1}) \cdots (\beta_p \alpha \beta_p^{-1}) = 1$.
We say that~$G$ has \emph{generalized torsion} if it contains a non-trivial generalized torsion element.
For most of our cases, the criterion we will use to show that~$A [\Gamma]$ and~$A [A_n]$ are not commensurable is given by the following two results.

\begin{lemma}\label{notinjective}
Let~$\Gamma$ be a connected Coxeter graph of spherical type with~$n$ vertices.
Let $\Phi: \overline{A [\Gamma]} \to \MM^* (\Sigma_{0,0}, \PP_{n+2})$ be a homomorphism and set $\varphi= \hat{\theta} \circ \Phi: \overline{A[\Gamma]}\to \SSS_{n+2}\times \{\pm 1\} $. 
If $\Ker (\varphi)$ has generalized torsion, then $\Phi$ is not injective.
\end{lemma}

\begin{proof}
Assume that~$\Phi$ is injective and that~$\Ker (\varphi)$ has generalized torsion.
As $\overline{\CA [A_n]} = \Ker (\hat{\theta})$, the homomorphism $\Phi$ induces an injective homomorphism $\Phi': \Ker (\varphi) \to \overline{\CA [A_n]}$. Recall that a group is called \emph{biorderable} if it admits a total ordering invariant under left-multiplication and right-multiplication.
We know by  \citep{RolZhu1} that~$\CA [A_n]$ is biorderable. 
By \autoref{tecnico}, $\overline{\CA [A_n]}$ is a subgroup of~$\CA [A_n]$, hence~$\overline{\CA [A_n]}$ is also biorderable, having that $\Ker (\varphi)$ is biorderable. However, a non-trivial biorderable group has no generalized torsion \citep{RolZhu1}. This is a contradiction. 
\end{proof}

\begin{corollary}\label{corolariocontraejemplos}
Let~$\Gamma$ be a connected Coxeter graph of spherical type with~$n$ vertices.
If the kernel of every homomorphism $\varphi : \overline{A [\Gamma]} \to \SSS_{n+2} \times \{ \pm 1\}$ has generalized torsion, then~$A [\Gamma]$ and~$A[A_n]$ are not commensurable.
\end{corollary}

\begin{proof}
Assume that~$A [\Gamma]$ and~$A [A_n]$ are commensurable. 
By \autoref{tecnico}, $\overline{A [\Gamma]}$ injects in $\Com(\overline{A [\Gamma]})$.
Again by \autoref{tecnico}, $\overline{A [\Gamma]}$ and~$\overline{A [A_n]}$ are commensurable, so $\Com (\overline{A [\Gamma]}) \simeq \Com (\overline{A [A_n]})$.
Moreover, by \autoref{CharneyCrisp}, $\Com (\overline{ A[A_{n}]}) = \MM^*(\Sigma_{0,0}, \PP_{n+2})$.
Then we have an injective homomorphism $\Phi : \overline{A [\Gamma]} \to \MM^* (\Sigma_{0,0}, \PP_{n+2})$.
Let $\varphi = \hat{\theta} \circ \Phi : \overline{A [\Gamma]} \to \SSS_{n+2} \times \{ \pm 1\}$. 
Therefore, by \autoref{notinjective}, $\Ker (\varphi)$ does not have generalized torsion, having a contradiction.
\end{proof}

\bigskip\noindent
From here, in order to finish our proof, for each considered Coxeter graph~$\Gamma$ and each homomorphism $\varphi\colon \overline{A[\Gamma]}\rightarrow \SSS_{n+2} \times \{\pm 1\}$, we show an element in $\Ker(\varphi)$ having generalized torsion. To find such elements we apply the following strategy, well-known to some experts. Let~$G$ be a group. An element $\beta\in G$ is \emph{quasi-central} if there exists $n\geq 1$ such that~$\beta^n$ lies in the center of~$G$. Assume that~$\beta$ is a quasi-central element and that~$\alpha$ is an element of~$G$ which does not commute with $\beta$. Then $\alpha \beta \alpha^{-1} \beta^{-1}$ is a generalized torsion non-trivial element of~$G$. The quasi-central elements of the Artin groups of spherical type are well-understood, and we look for quasi-central elements in standard parabolic subgroups that belong to the kernel of $\varphi$ to find our generalized torsion elements.

\medskip\noindent
We will use the following notations and definitions. For a group~$G$ and $\alpha \in G$ we denote by $c_\alpha : G \to G$, $\beta \mapsto \alpha \beta \alpha^{-1}$, the conjugation by~$\alpha$.
We say that two homomorphisms $\varphi_1, \varphi_2 : G \to H$ are \emph{conjugate} if there is $\alpha \in H$ such that $\varphi_2 = c_\alpha \circ \varphi_1$.
Moreover, a homomorphism $\varphi : G \to H$ is said to be \emph{cyclic} if the image of~$\varphi$ is a cyclic subgroup of~$H$.

\begin{lemma}
The groups $A[D_4]$ and $A[A_4]$ are not commensurable.
\end{lemma}

\begin{proof}
Let $\varphi : \overline{A[D_4]} \to \SSS_6\times \{ \pm 1\}$ be a homomorphism written in the form $\varphi=\varphi_1 \times \varphi_2$, where $\varphi_1 : \overline{A[D_4]} \to \SSS_6$ and $\varphi_2 : \overline{A[D_4]} \to  \{ \pm 1 \}$ are two homomorphisms.
By \autoref{corolariocontraejemplos}, we just need to show that~$\Ker (\varphi)$ has generalized torsion.
We denote by $s_1, s_2, s_3, s_4$ the standard generators of~$A[D_4]$ numbered as in \autoref{coxeter}.
We also denote by $\pi: A [D_4] \to \overline{A[D_4]}$ the quotient homomorphism and $\bar s_i = \pi (s_i)$ for every $i \in \{ 1,2,3,4 \}$.
Notice that $\varphi_2$ is always cyclic since its image is contained in $\{\pm 1\}$, which is a cyclic group. Then, the relations $s_js_3s_j =s_3s_js_3$, for $j\in \{1,2,4\}$, imply that there is $\epsilon \in \{ \pm 1 \}$ such that $\varphi_2(\bar s_i) = \epsilon$ for every $i \in \{1,2,3,4 \}$.

\bigskip\noindent
Firstly, suppose that~$\varphi_1$ is cyclic. Let $\alpha = \bar s_1 \bar s_2^{-1}$ and $\beta = \bar s_3 \bar s_2 \bar s_1 \bar s_3 \bar s_1^{-4}$.
Then $\alpha, \beta \in \Ker (\varphi)$, $\alpha \neq 1$, and $\alpha \beta \alpha \beta^{-1} = 1$, having that~$\Ker (\varphi)$ has generalized torsion. 
Now, suppose that $\varphi_1$ is not cyclic. A direct computation using the software SageMath (see code in \citealp{code}) shows that there are $14400$ non-cyclic homomorphisms from~$\overline{A[D_4]}$ to~$\SSS_6$ divided into $40$~conjugacy classes.
By using the same software, we check that in every case we have either $\varphi_1 (\bar s_1) = \varphi_1 (\bar s_2)$ or $\varphi_1 (\bar s_1) = \varphi_1 (\bar s_4)$ or $\varphi_1 (\bar s_2) = \varphi_1 (\bar s_4)$.
Then we can assume without loss of generality that $\varphi_1 (\bar s_1) = \varphi_1 (\bar s_2)$.
In this case we have $8640$~homomorphisms satisfying our conditions that are divided into $24$~conjugacy classes.
Let $\beta = \bar s_1 \bar s_3 \bar s_2 \bar s_1 \bar s_3 \bar s_1$ and $\alpha = \bar s_1 \bar s_2^{-1}$. Note that they both belong to $\Ker(\varphi_2)$.
We check that $\varphi_1(\beta) = 1$ in every case.
Moreover, as $\varphi_1 (\bar s_1) = \varphi_1 (\bar s_2)$, we also have that $\varphi_1(\alpha)=1$. Therefore, $\alpha, \beta \in \Ker (\varphi)$, $\alpha \neq 1$ and $\alpha \beta \alpha \beta^{-1} = 1$ and $\Ker (\varphi)$ has generalized torsion.
\end{proof}

\begin{lemma}
Let $n \ge 5$.
Then $A [D_n]$ and $A [A_n]$ are not commensurable.
\end{lemma}

\begin{proof}
We denote by $s_1, \dots, s_n$ the standard generators of~$A[D_n]$ numbered as in \autoref{coxeter}.
We also let $t_i = (i,i+1) \in \SSS_{n+2}$ for every $i \in \{1, \dots, n+1\}$.
Let $\zeta : A[D_n] \to \SSS_{n+2}$ be the homomorphism defined by $\zeta (s_1) = \zeta (s_2) = t_1$ and $\zeta (s_i) = t_{i-1}$ for every $i \in \{3, \dots, n\}$.
Moreover, for $n=6$, let $\nu : A[D_6] \to \SSS_8$ be the homomorphism defined by
\begin{gather*}
\nu (s_1) = \nu (s_2) = (1, 2)(3, 4)(5, 6)\,,\
\nu (s_3) = (2, 3)(1, 5)(4, 6)\,,\
\nu (s_4) = (1, 3)(2, 4)(5, 6)\,,\\
\nu (s_5) = (1, 2)(3, 5)(4, 6)\,,\
\nu (s_6) = (2, 3)(1, 4)(5, 6)\,.
\end{gather*}

\bigskip\noindent
{\it Claim.}
Let $\psi : A [D_n] \to \SSS_{n+2}$ be a homomorphism.
Then, we have one of the following situations, up to conjugation. 
\begin{itemize}
\item[(1)]
$\psi$ is cyclic,
\item[(2)]
$\psi = \zeta$,
\item[(3)]
$n=6$ and $\psi = \nu$.
\end{itemize}

\bigskip\noindent
{\it Proof of the claim.}
Let $s_1', \dots, s_{n-1}'$ be the standard generators of~$A [A_{n-1}]$.
Let $\zeta' : A [A_{n-1}] \to \SSS_{n+2}$ be the homomorphism defined by $\zeta' (s_i') =t_i$ for every $i \in \{1, \dots, n-1 \}$.
For $n=6$, let $\nu' : A[A_{5}] \to \SSS_8$ be the homomorphism defined by
\begin{gather*}
\nu'(s_1') = (1, 2)(3, 4)(5, 6)\,,\
\nu' (s_2') = (2, 3)(1, 5)(4, 6)\,,\
\nu'(s_3') = (1, 3)(2, 4)(5, 6)\,,\\
\nu'(s_4') = (1, 2)(3, 5)(4, 6)\,,\
\nu'(s_5') = (2, 3)(1, 4)(5, 6)\,.
\end{gather*}
Let $\iota : A[A_{n-1}] \to A[D_n]$ be the homomorphism sending $s_i'$ to $s_{i+1}$ for every $i \in \{1, \dots, n-1\}$, and $\psi' = \psi \circ \iota : A[A_{n-1}] \to \SSS_{n+2}$.
By \citep[Theorem~1]{Artin2} and \citep[Theorem~A, Theorem~E]{Lin1}, we have one of the following possibilities, up to conjugation. 
\begin{itemize}
\item[(1)]
$\psi'$ is cyclic,
\item[(2)]
$\psi' = \zeta'$,
\item[(3)]
$n=6$ and $\psi' = \nu'$.
\end{itemize}

\bigskip\noindent
First assume that~$\psi'$ is cyclic.
Then there is $w \in \SSS_{n+2}$ such that $\psi' (s_i') = \psi(s_{i+1}) = w$ for every $i \in \{1, \dots, n-1\}$.
Let $\gamma = s_1 s_3 s_2 s_1 s_3 s_1$.
We have $\gamma s_2 \gamma^{-1} = s_1$ and $\gamma s_3 \gamma^{-1} = s_3$, hence $w = \psi(s_3) = \psi (\gamma s_3 \gamma^{-1}) = \psi (\gamma s_2 \gamma^{-1}) = \psi (s_1)$.
Thus, $\psi$ is cyclic.

\bigskip\noindent
Now suppose that $\psi' = \zeta'$.
We have $\psi (s_{i+1}) = \psi' (s_i') = t_i$ for every $i \in \{1, \dots, n-1\}$.
Let $u = \psi (s_1)$.
As~$u$ commutes with $\psi (s_i) = t_{i-1}$ for every $i \ge 4$, it follows that~$u(k)=k$ for every $k \in \{3, 4, \dots, n \}$.
Moreover, $u$~commutes with $t_1 = \psi (s_2)$, so $u \in E = \{1, t_1, t_{n+1}, t_1 t_{n+1} \}$.
The only element~$u$ of~$E$ satisfying $u t_2 u = t_2 u t_2$ is $u=t_1$, hence $u=t_1$ and $\psi = \zeta$.

\bigskip\noindent
Assume that~$n=6$ and $\psi' = \nu'$.
Let
\begin{gather*}
u_1 = (1, 2)(3, 4)(5, 6)\,,\
u_2 = (2, 3)(1, 5)(4, 6)\,,\
u_3 = (1, 3)(2, 4)(5, 6)\,,\\
u_4 = (1, 2)(3, 5)(4, 6)\,,\
u_5 = (2, 3)(1, 4)(5, 6)\,.
\end{gather*}
A direct computation with the software SageMath (see code in \citealp{code}) shows that the only element $v \in \SSS_8$ satisfying $v u_1 = u_1 v$, $v u_2 v = u_2 v u_2$, $v u_3 = u_3 v$, $v u_4 = u_4 v$, $v u_5 = u_5 v$ is $v=u_1$, hence $\psi = \nu$.
This finishes the proof of the claim.

\bigskip\noindent
Let~$\Delta$ be the Garside element of~$A[D_n]$ and let~$\delta$ be the standard generator of~$Z (A [D_n])$.
By \citep[Lemma~5.1]{Paris4}, $\Delta = (s_n \cdots s_3 s_2 s_1 s_3 \cdots s_n) \cdots (s_3 s_2 s_1 s_3)(s_2 s_1)$. Moreover, $\delta=\Delta$ if $n$ is even, and $\delta=\Delta^2$ si $n$ is odd.
Notice that $\zeta (\Delta) = 1$, so $\zeta (\delta) = 1$. 
It follows that $\zeta$ induces a homomorphism~$\bar \zeta: \overline{A[D_n]} \to \SSS_{n+2}$.
Similarly, if $n=6$, $\nu (\Delta) = 1$ and $\nu (\delta) = 1$, then $\nu$ induces a homomorphism $\bar \nu : \overline{A [D_6]} \to \SSS_8$.
Let $\varphi_1 : \overline{A [D_n]} \to \SSS_{n+2}$ be a homomorphism.
Then, by the results from above and the claim, we have one of the following three possibilities, up to conjugation. 
\begin{itemize}
\item[(1)]
$\varphi_1$ is cyclic,
\item[(2)]
$\varphi_1 = \bar \zeta$,
\item[(3)]
$n=6$ and $\varphi_1 = \bar \nu$.
\end{itemize}

\bigskip\noindent
Let $\varphi : \overline{A[D_n]} \to \SSS_{n+2}\times \{ \pm 1 \}$ be a homomorphism written in the form $\varphi = \varphi_1 \times \varphi_2$, where ${\varphi_1 : \overline{A[D_n]} \to \SSS_{n+2}}$ and $\varphi_2 : \overline{A[D_n]} \to  \{ \pm 1 \}$ are homomorphisms.
By \autoref{corolariocontraejemplos}, we just need to show that $\Ker (\varphi)$ has generalized torsion. Denote by $\pi : A [D_n] \to \overline{A [D_n]}$ the quotient homomorphism and $\bar s_i = \pi(s_i)$ for every $i \in \{1, \dots, n\}$. Here again, $\varphi_2$ is always cyclic since its image is contained in $\{\pm 1\}$.

\bigskip\noindent
Suppose that $\varphi_1$ is cyclic.
Let $\alpha = \bar s_1 \bar s_2^{-1}$ and $\beta =  \bar s_3 \bar s_2 \bar s_1 \bar s_3 \bar s_1^{-4}$.
In this case $\alpha, \beta \in \Ker (\varphi)$, $\alpha \neq 1$ and $\alpha \beta \alpha \beta^{-1} =1$, and then $\Ker (\varphi)$ has generalized torsion.
Assume either $\varphi = \bar \zeta$ or $n=6$ and $\varphi = \bar \nu$.
Let $\alpha = \bar s_1 \bar s_2^{-1}$ and $\beta = \bar s_1 \bar s_3 \bar s_2 \bar s_1 \bar s_3 \bar s_1$.
In both cases $\alpha, \beta \in \Ker (\varphi)$, $\alpha \neq 1$ and $\alpha \beta \alpha \beta^{-1} =1$, hence $\Ker (\varphi)$ has generalized torsion.
\end{proof}

\begin{lemma}
Let $n \in \{6,7,8\}$.
Then $A [E_n]$ and $A [A_n]$ are not commensurable.
\end{lemma}

\begin{proof}
We denote by $s_1, \dots, s_n$ the standard generators of~$A [E_n]$ numbered as in \autoref{coxeter}.
We also let $t_i = (i,i+1) \in \SSS_{n+2}$ for every $i \in \{1, \dots, n+1\}$.

\bigskip
\noindent
{\it Claim.}
Every homomorphism $\psi : A[E_n] \to \SSS_{n+2}$ is cyclic.

\bigskip
\noindent
{\it Proof of the claim.}
Denote by $s_1', \dots, s_{n-1}'$ the standard generators of~$A [A_{n-1}]$.
Let $\zeta' :A [A_{n-1}] \to \SSS_{n+2}$ be the homomorphism defined by $\zeta' (s_i') =t_i$ for every $i \in \{1, \dots, n-1 \}$.
For $n=6$, let $\nu' : A[A_{5}] \to \SSS_{8}$ be the homomorphism defined by
\begin{gather*}
\nu'(s_1') = (1, 2)(3, 4)(5, 6)\,,\
\nu' (s_2') = (2, 3)(1, 5)(4, 6)\,,\
\nu'(s_3') = (1, 3)(2, 4)(5, 6)\,,\\
\nu'(s_4') = (1, 2)(3, 5)(4, 6)\,,\
\nu'(s_5') = (2, 3)(1, 4)(5, 6)\,.
\end{gather*}
Let $\iota : A[A_{n-1}] \to A[E_n]$ be the homomorphism sending  $s_i'$ to $s_i$ for every $i \in \{1, \dots, n-1\}$, and $\psi' = \psi \circ \iota : A[A_{n-1}] \to \SSS_{n+2}$.
By \citep[Theorem~1]{Artin2} and \citep[Theorem~A, Theorem~E]{Lin1}, we have one of the following possibilities, up to conjugation. 
\begin{itemize}
\item[(1)]
$\psi'$ is cyclic,
\item[(2)]
$\psi' = \zeta'$,
\item[(3)]
$n=6$ and $\psi' = \nu'$.
\end{itemize}

\bigskip\noindent
First suppose that~$\psi'$ is cyclic.
Then there is $w \in \SSS_{n+2}$ such that $\psi' (s_i') = \psi (s_i) = w$ for every $i \in \{1, \dots, n-1\}$.
Let $\gamma = s_2 s_3 s_n s_2 s_3 s_2$.
We have $\gamma s_2 \gamma^{-1} = s_n$ and $\gamma s_3 \gamma^{-1} = s_3$, hence $w = \psi (s_3) = \psi (\gamma s_3 \gamma^{-1}) = \psi (\gamma s_2 \gamma^{-1}) = \psi (s_n)$.
Then~$\psi$ is cyclic.

\bigskip\noindent
Now assume that $\psi' = \zeta'$.
In this case we have $\psi (s_i) = \psi' (s_i') = t_i$ for every $i \in \{1, \dots, n-1\}$.
Let $u = \psi(s_n)$.
As~$u$ commutes with $\psi (s_i) = t_i$ for every $i \in \{1, 2, 4, \dots, n-1\}$, it follows that $u(k) =k$ for every $k \in \{1,2,3,4,5, \dots, n\}$, so $u \in E = \{1, t_{n+1}\}$.
But there is no element~$u$ of~$E$ satisfying $u t_3 u = t_3 u t_3$, so we cannot have $\psi' = \zeta'$.

\bigskip\noindent
Finally, assume $n=6$ and $\psi' = \nu'$.
Let
\begin{gather*}
u_1 = (1, 2)(3, 4)(5, 6)\,,\
u_2 = (2, 3)(1, 5)(4, 6)\,,\
u_3 = (1, 3)(2, 4)(5, 6)\,,\\
u_4 = (1, 2)(3, 5)(4, 6)\,,\
u_5 = (2, 3)(1, 4)(5, 6)\,.
\end{gather*}
A direct computation with the software SageMath (see code in \citealp{code}) shows that there is no element $v \in \SSS_8$ satisfying $v u_1 = u_1 v$, $v u_2 = u_2 v$, $v u_3 v = u_3 v u_3$, $v u_4 = u_4 v$ and $v u_5 = u_5 v$, hence we cannot have $n=6$ and $\psi' = \nu'$.
This finishes the proof of the claim.

\bigskip\noindent
Denote by $\pi : A [E_n] \to \overline{A [E_n]}$ the quotient homomorphism and $\bar s_i = \pi(s_i)$ for every $i \in \{1, \dots, n\}$.
Let $\varphi: \overline{A[E_n]} \to \SSS_{n+2}\times \{ \pm 1 \}$ be a homomorphism written in the form $\varphi = \varphi_1 \times \varphi_2$, where $\varphi_1 :\overline{A[E_n]} \to \SSS_{n+2}$ and $\varphi_2 : \overline{A[E_n]} \to \{ \pm 1 \}$ are homomorphisms.
By the claim, $\varphi_1 \circ \pi : A[E_n] \to \SSS_{n+2}$ is cyclic, hence $\varphi_1$ is also cyclic. On the other hand, $\varphi_2$ is cyclic since the image of $\varphi_2$ is contained in $\{\pm 1\}$.
Let $\alpha = \bar s_2 \bar s_n^{-1}$ and $\beta =  \bar s_3 \bar s_n \bar s_2 \bar s_3 \bar s_2^{-4}$.
We have that $\alpha, \beta \in \Ker (\varphi)$, $\alpha \neq 1$ and $\alpha \beta \alpha \beta^{-1} =1$, and then $\Ker (\varphi)$ has generalized torsion.
By \autoref{corolariocontraejemplos}, it follows that~$A [A_n]$ and~$A[E_n]$ are not commensurable.
\end{proof}

\begin{lemma}
The groups $A[F_4]$ and $A[A_4]$ are not commensurable.
\end{lemma}

\begin{proof}
Let $\varphi : \overline{A[F_4]} \to \SSS_6\times \{ \pm 1\}$ be a homomorphism written in the form $\varphi=\varphi_1 \times \varphi_2$, where $\varphi_1 : \overline{A[F_4]} \to \SSS_6$ and $\varphi_2 : \overline{A[F_4]} \to  \{ \pm 1 \}$ are two homomorphisms.
By \autoref{corolariocontraejemplos}, we just need to show that~$\Ker (\varphi)$ has generalized torsion.
We denote by $s_1, s_2, s_3, s_4$ the standard generators of~$A[F_4]$ numbered as in \autoref{coxeter}. 
We also denote by $\pi: A [F_4] \to \overline{A[F_4]}$ the quotient homomorphism and $\bar s_i = \pi (s_i)$ for every $i \in \{ 1,2,3,4 \}$.
Notice that the relation $s_1 s_2 s_1 = s_2 s_1 s_2$ implies $\varphi_2(\bar s_1) = \varphi_2(\bar s_2)$.
Analogously, $\varphi_2(\bar s_3) = \varphi_2(\bar s_4)$.

\bigskip\noindent
If $g$ is an element of a group, we denote by~$\ord(g)$ the order of~$g$.
Let $i \in \{1,2,3,4\}$.
As $\varphi_1(\bar s_i) \in \SSS_6$, we have that $\ord(\varphi_1(\bar s_i)) \in \{1,2,3,4,5,6\}$.
It follows that $\ord(\varphi(\bar s_i)) \in \{1,2,3,4,5,6,10\}$.
Suppose that $\varphi(\bar s_1)=\varphi(\bar s_2)$ and $\ord(\varphi(\bar s_1)) \in \{1,2,4\}$.
Let $\alpha = \bar s_1 \bar s_2^{-1}$ and $\beta = (\bar s_1 \bar s_2)^4$.
In this case $\alpha, \beta \in \Ker(\varphi)$, $\alpha \neq 1$ and $\alpha (\beta \alpha \beta^{-1}) (\beta^2 \alpha \beta^{-2}) =1$, and then $\Ker (\varphi)$ has generalized torsion.
Now assume that $\varphi(\bar s_1)=\varphi(\bar s_2)$ and $\ord(\varphi(\bar s_1)) =3$.
If we let $\alpha = \bar s_1 \bar s_2^{-1}$, $\beta = \bar s_1 \bar s_2 \bar s_1$, then $\alpha, \beta \in \Ker (\varphi)$, $\alpha \neq 1$, $\alpha (\beta \alpha \beta^{-1}) =1$, and $\Ker (\varphi)$ has generalized torsion.
Now suppose that $\varphi(\bar s_1)=\varphi(\bar s_2)$ and $\ord(\varphi(\bar s_1)) \in \{5,10\}$.
We let $\alpha = \bar s_1 \bar s_2^{-1}$ and $\beta = (\bar s_1 \bar s_2)^{10}$. Then $\alpha, \beta \in \Ker (\varphi)$, $\alpha (\beta \alpha \beta^{-1}) (\beta^2 \alpha \beta^{-2}) =1$, and $\Ker (\varphi)$ has generalized torsion.

\bigskip\noindent
By the reasoning above we can assume that, if $\varphi(\bar s_1) = \varphi(\bar s_2)$, then $\ord(\varphi(\bar s_1)) = 6$.
We can also suppose that, if $\varphi(\bar s_3) = \varphi(\bar s_4)$, then $\ord(\varphi(\bar s_3)) = 6$.

\bigskip\noindent
Suppose that $\varphi(\bar s_1) = \varphi(\bar s_2)$ and $\varphi(\bar s_3) = \varphi(\bar s_4)$. Then we also have $\ord(\varphi(\bar s_1)) = \ord(\varphi(\bar s_3)) = 6$.
If $\varphi_1(\bar s_1) = \varphi_1(\bar s_2)$ and $\varphi_1(\bar s_3) = \varphi_1(\bar s_4)$ are both of order~$3$, then $\varphi_2(\bar s_1) = \varphi_2(\bar s_2) = \varphi_2 (\bar s_3) = \varphi_ 2(\bar s_4) = -1$.
In this case, we let $\alpha = \bar s_1 \bar s_2^{-1}$ and $\beta = \bar s_1 \bar s_2 \bar s_1 \bar s_4^3$, having $\alpha, \beta \in \Ker (\varphi)$, $\alpha \neq 1$ and $\alpha (\beta \alpha \beta^{-1}) = 1$. Hence $\Ker (\varphi)$ has generalized torsion.
We can then assume that $\varphi_1(\bar s_1)$ or~$\varphi_1(\bar s_3)$ is of order~$6$, say that~$\varphi_1(\bar s_1)$ has order~$6$.
Then $\varphi_1(\bar s_1)$ is conjugate to $(1,2,3,4,5,6)$ or to $(1,2,3)(4,5)$ in~$\SSS_6$.
In both cases it follows that the centralizer of~$\varphi_1(\bar s_1)$ in~$\SSS_6$ is a cyclic group of order~$6$ generated by~$\varphi_1(\bar s_1)$.
As~$\varphi_1(\bar s_3)$ belongs to this centralizer and it has order~$3$ or~$6$, there is $k \in \{1,2,-1,-2\}$ such that $\varphi_1(\bar s_3)=\varphi_1(\bar s_4)=\varphi_1(\bar s_1)^k$.
We let $\alpha = \bar s_3 \bar s_4^{-1}$ and $\beta = \bar s_3 \bar s_4 \bar s_1^{-2k}$.
Then, $\alpha, \beta \in \Ker (\varphi)$, $\alpha \neq 1$, and $\alpha (\beta \alpha \beta^{-1}) (\beta^2 \alpha \beta^{-2}) = 1$,  having generalized torsion in $\Ker (\varphi)$.

\bigskip\noindent
By \citep{BrieskornSaito}, the standard generator of the center of~$A[F_4]$ coincides with its Garside element and equals $(s_1 s_2 s_3 s_4)^{\frac{h}{2}}$ where $h$ is the Coxeter number of~$F_4$.
As $h=12$ \citep[Page~80]{Humph1}, we have $\delta = \Delta = (s_1 s_2 s_3 s_4)^6$.
Let $\hat\alpha_0 = (s_1 s_2 s_3 s_4)^3$.
Recall that $z : A [F_4] \to \Z$ is the homomorphism sending~$s_i$ to~$1$ for every $i \in \{1,2,3,4\}$.
As $z(\delta) = 24$, we have $z(Z(A[F_4]))=24\Z$, so $\hat \alpha_0 \not\in Z(A[F_4])$ because $z(\hat \alpha_0) =12$.
On the other hand, $\hat \alpha_0^2=\delta$, so $\hat \alpha_0^2\in Z(A[F_4])$. Let $\alpha_0=\pi(\hat \alpha_0 )$. Then $\alpha_0\neq 1$ and $\alpha_0^2=1$. In the remaining cases, we will show that $\alpha_0\in \Ker(\varphi)$, which will immediately imply that $\Ker(\varphi)$ has (generalized) torsion.

\bigskip\noindent
Suppose that $\varphi (\bar s_1) \neq \varphi (\bar s_2)$ and $\varphi (\bar s_3) = \varphi(\bar s_4)$ (hence $\ord(\varphi(\bar s_3)) = 6$).
Let~$E_1$ be the set of triples $(u_1,u_2,u_3)$ of elements of~$\SSS_6$ such that $u_1 u_2 u_1 = u_2 u_1 u_2$, $u_1 u_3 = u_3 u_1$, $u_2 u_3 = u_3 u_2$, $u_1 \neq u_2$ and $\ord (u_3) \in \{3,6\}$.
Another direct computation with SageMath (see code in \citealp{code}) shows that~$E_1$ has $1440$~elements divided into $6$~conjugacy classes.
Again with SageMath, we compute a set~$E_1^0$ of representatives of the conjugacy classes in~$E_1$ and we get
\begin{gather*}
E_1^0 = \big\{
\big((1,2), (2,3), (4,5,6)\big),
\big((1,2,3,4,5,6), (1,6,3,2,5,4), (1,3,5)(2,4,6)\big),\\
\big((1,2,3,4,5,6), (1,6,3,2,5,4), (1,5,3)(2,6,4)\big),\\
\big((1,4)(2,5)(3,6), (1,2)(3,4)(5,6), (1,3,5)(2,4,6)\big),\\
\big((2,3)(4,5,6), (1,2)(4,5,6), (4,5,6)\big),
\big((2,3)(4,5,6), (1,2)(4,5,6), (4,6,5)\big)\big\}\,.
\end{gather*}
We check with a direct computation that $(u_1 u_2 u_3^2)^3=1$ for every $(u_1,u_2,u_3) \in E_1^0$.
Up to conjugation, we can suppose that $(\varphi_1(\bar s_1), \varphi_1 (\bar s_2), \varphi_1 (\bar s_3))=(u_1,u_2,u_3) \in E_1^0$.
Then, as $(u_1 u_2 u_3^2)^3=1$, we have $\varphi_1(\alpha_0)=1$.
It is obvious that $\varphi_2 (\alpha_0) = 1$. So, $\varphi(\alpha_0) = 1$ and $\Ker (\varphi)$ has (generalized) torsion. 

\bigskip\noindent
Suppose that $\varphi (\bar s_1) \neq \varphi (\bar s_2)$ and $\varphi (\bar s_3) \neq \varphi(\bar s_4)$.  Let~$E_2$ be the set  of quadruples $(u_1,u_2,u_3,u_4)$ of elements of~$\SSS_6$ such that $u_1 u_2 u_1 = u_2 u_1 u_2$, $u_1 u_3 = u_3 u_1$, $u_1 u_4 = u_4 u_1$, $u_2 u_3 u_2 u_3 = u_3 u_2 u_3 u_2$, $u_2 u_4 = u_4 u_2$, $u_3 u_4 u_3 = u_4 u_3 u_4$, $u_1 \neq u_2$ and $u_3 \neq u_4$.
A direct computation with SageMath (see code in \citealp{code}) shows that~$E_2$ has $1440$~elements divided into $2$~conjugacy classes.
Again with SageMath, we compute a set~$E_2^0$ of representatives of the conjugacy classes in~$E_2$ and we get
\begin{gather*}
E_2^0 = \big\{
\big( (1,2), (2,3), (5,6), (4,5) \big),\\
\big( (1,4)(2,5)(3,6), (1,2)(3,4)(5,6), (1,4)(2,3)(5,6), (1,6)(2,5)(3,4) \big) \big\}\,.
\end{gather*}
We check by a direct computation that $(u_1 u_2 u_3 u_4)^3=1$ for every $(u_1,u_2,u_3,u_4) \in E_2^0$.
Up to conjugation, we can suppose that $(\varphi_1(\bar s_1), \varphi_1 (\bar s_2), \varphi_1 (\bar s_3), \varphi_1 (\bar s_4))=(u_1,u_2,u_3,u_4) \in E_2^0$.
Then, as $(u_1 u_2 u_3 u_4)^3=1$, we have $\varphi_1(\alpha_0)=1$.
It is clear that $\varphi_2 (\alpha_0) = 1$. Then $\varphi(\alpha_0) = 1$ and $\Ker (\varphi)$ has (generalized) torsion. 
\end{proof}

\begin{lemma}
{The groups $A [H_4]$ and $A [A_4]$ are not commensurable.}
\end{lemma}

\begin{proof}
We denote by $s_1, s_2, s_3, s_4$ the standard generators of~$A[H_4]$ numbered as in \autoref{coxeter}.
We also consider the quotient homomorphism $\pi: A [H_4] \to \overline{A[H_4]}$ and $\bar s_i = \pi (s_i)$ for every $i \in \{ 1,2,3,4 \}$.
Let $\varphi: \overline{A[H_4]} \to \SSS_{6}\times \{\pm 1\}$ be a homomorphism written in the form $\varphi = \varphi_1 \times \varphi_2$ where $\varphi_1 : \overline{A[H_4]} \to \SSS_{6}$ and $\varphi_2: \overline{A[H_4]} \to \{ \pm 1 \}$ are two homomorphisms.
Notice that the relations $s_3s_4s_3=s_4s_3s_4$, $s_2 s_3 s_2 = s_3 s_2 s_3$, $s_1 s_2 s_1 s_2 s_1 =s_2 s_1 s_2 s_1 s_2$ imply $\varphi_2(\bar s_1)=\varphi_2(\bar s_2) = \varphi_2(\bar s_3)=\varphi_2(\bar s_4)$. Then~$\varphi_2$ is always cyclic. For $\varphi_1$, a direct computation with SageMath (see code in \citealp{code}) shows that there are $720$~homomorphisms from~$\overline{A[H_4]}$ to~$\SSS_6$, all of them cyclic. 
We let $\alpha=\bar s_3^{-1}\bar s_4$ and $\beta=\bar s_3\bar s_4\bar s_3\bar s_1^{-3}$. They both belong to $\Ker(\varphi)$ and they satisfy $\alpha\beta\alpha\beta^{-1}=1$. Therefore, $\Ker(\varphi)$ has generalized torsion and, by \autoref{corolariocontraejemplos}, $A[A_4]$ and~$A[H_4]$ are not commensurable.
\end{proof}

\noindent
Our last issue is to compare~$A [H_3]$ and~$A[A_3]$. In this case we cannot apply \autoref{corolariocontraejemplos}, as we have done with the previous cases. Indeed, we can find homomorphisms sending $\overline{A[H_3]}$ to $\SSS_5$ whose kernel does not have generalized torsion:

\begin{lemma} Let $s_1, s_2, s_3$ be the standard generators of $A[H_3]$ numbered as in \autoref{coxeter}, let $\pi\colon A[H_3]\to \overline{A[H_3]}$ be the quotient homomorphism, and, for each $i\in\{1,2,3\}$, let $\bar s_i = \pi(s_i)$.
Let $\zeta : \overline{A[H_3]} \to \SSS_5$ be the homomorphism defined by
\[
\zeta (\bar s_1) = (2,4)(3,5)\,,\ 
\zeta (\bar s_2) = (1,2)(4,5)\,,\
\zeta (\bar s_3) = (2,3)(4,5)\,.
\]
Then $\Ker (\zeta)$ does not have generalized torsion.
\end{lemma}

\begin{proof}

By \citep{BrieskornSaito}, $Z(A[H_3])$ is an infinite cyclic group generated by $\delta = (s_1 s_2 s_3)^5$.
Let $u_1 = (2,4)(3,5)$, $u_2 = (1,2)(4,5)$ and $u_3 = (2,3)(4,5)$.
A direct computation shows that we have the relations $u_1 u_2 u_1 u_2 u_1 = u_2 u_1 u_2 u_1 u_2$, $u_1 u_3 = u_3 u_1$, $u_2 u_3 u_2 = u_3 u_2 u_3$ and $(u_1 u_2 u_3)^5=1$, hence $\zeta$ is well-defined.
We are going to prove that $\Ker (\zeta) = \overline{\CA [H_3]}$.
As $\overline{\CA [H_3]}$ projects into $\CA [H_3]$ by \autoref{tecnico} and $\CA [H_3]$ has no generalized torsion by \citep[Theorem~3]{Mar}, it will follow that $\Ker (\zeta)$ has no generalized torsion.

\medskip\noindent
Let~$H$ be the subgroup of~$\SSS_5$ generated by $\{ u_1, u_2, u_3 \}$.
A direct computation with SageMath (see code in \citealp{code}) shows that $|H| = 60$.
As $u_1^2 = u_2^2 = u_3^2 = 1$ and $\overline{\CA [H_3]}$ is the normal subgroup of~$\overline{A [H_3]}$ generated by $\{\bar s_1^2, \bar s_2^2, \bar s_3^2 \}$, we have $\overline{\CA [H_3]} \subset \Ker (\zeta)$.
Then, to show that $\Ker (\zeta) =\overline{\CA [H_3]}$, we just need to prove that $|\overline{A [H_3]}/\overline{\CA [H_3]}| = 60$.
It is well known that $|A[H_3]/\CA [H_3]| = |W[H_3]| = 120$  \citep[Page~46]{Humph1}.
The projection $\pi : A [H_3] \to \overline{A [H_3]}$ induces a surjective homomorphism $\bar \pi : A[H_3]/\CA [H_3] \to \overline{A [H_3]}/\overline{\CA [H_3]}$ whose kernel is the cyclic group generated by the class~$[\delta]$ of~$\delta$.
We have that $\delta = \Delta$ is the Garside element of~$A[H_3]$, so $\delta \not\in \CA [H_3]$. However, $\delta^2 = \Delta^2 \in \CA [H_3]$, hence $\Ker (\bar \pi)$ is a cyclic group~$\langle [\delta] \rangle$ of order~$2$, having $|\overline{A [H_3]}/\overline{\CA [H_3]}| = |A[H_3]/\CA [H_3]|/2 = 60$.
\end{proof}

\noindent
Let $\Sigma$ be a closed surface and $\mathcal{P}_n $ be a collection of $n$ different points in~$\Sigma$. With such a pair~$(\Sigma,\mathcal{P}_n)$ we can associate a simplicial complex called the \emph{curve complex} of $(\Sigma,\mathcal{P}_n)$, denoted by $\mathcal{C}(\Sigma,\mathcal{P}_n)$. The vertices of $\mathcal{C}(\Sigma,\mathcal{P}_n)$ are the isotopy classes of simple closed curves on~$\Sigma\setminus \PP_n$ that are non-degenerate. Non-degenerate means that the curve does not bound a disk embedded in $\Sigma$ containing 0 or 1 point of $\PP_n$. Every $n$-simplex is formed by $n+1$ classes having representatives that are pairwise disjoint. We say that a mapping class $f\in\MM^*(\Sigma,\mathcal{P}_n)$ is \emph{pseudo-Anosov} if $f^n(\alpha)\neq \alpha$ for every $\alpha\in\mathcal{C}(\Sigma,\mathcal{P}_n)$ and every $n\in \Z\setminus\{0\}$. We say that $f$ is \emph{periodic} if it has finite order.
The following lemma finishes the proof of \autoref{classification}. 

\begin{lemma}
The groups $A [H_3]$ and $A [A_3]$ are not commensurable.
\end{lemma}

\begin{proof}
Recall that, by \autoref{tecnico}, we need to prove that~$\overline{A [H_3]}$ and~$\overline{A[A_3]}$ are not commensurable, and, to do this, it is enough to prove that $\Com (\overline{A [H_3]})$ and $\Com (\overline{A [A_3]})$ are not isomorphic.
Also by \autoref{tecnico}, $\overline{A[H_3]}$ injects in $\Com (\overline{A [H_3]})$ and recall that $\Com (\overline{A [A_3]})$ and $\MM^*(\Sigma_{0,0}, \PP_5)$ are isomorphic by \autoref{CharneyCrisp}. 
Then, to prove our lemma it suffices to prove that there is no injective homomorphism from $\overline{A [H_3]}$ to $\MM^* (\Sigma_{0,0}, \PP_5)$.

\bigskip\noindent
Let $\Phi : \overline{ A [H_3]} \to \MM^* (\Sigma_{0,0},\PP_5)$ be a homomorphism. Recall $\hat \theta$ and $\theta'$ defined right before \autoref{CharneyCrisp} and
consider $\varphi = \hat \theta \circ \Phi : \overline{A [H_3]} \to \SSS_5 \times \{ \pm 1\}$ being of the form $\varphi = \varphi_1 \times \varphi_2$, where $\varphi_1 = \theta' \circ \Phi : \overline{A [H_3]} \to \SSS_5$ and $\varphi_2 = \omega \circ \Phi : \overline{A [H_3]} \to \{ \pm 1\}$.
We denote by $s_1, s_2, s_3$ the standard generators of~$A[H_3]$ numbered as in \autoref{coxeter}.
Moreover, we let $\pi: A [H_3] \to \overline{A[H_3]}$ be the quotient homomorphism and $\bar s_i = \pi (s_i)$ for every $i \in \{ 1,2,3\}$.

\bigskip\noindent
Notice that the relations $s_2 s_3 s_2 = s_3 s_2 s_3$ and $s_1 s_2 s_1 s_2 s_1 =s_2 s_1 s_2 s_1 s_2$ imply $\varphi_2(\bar s_2) = \varphi_2(\bar s_3)$ and $\varphi_2(\bar s_1)=\varphi_2(\bar s_2)$.  Notice also that the standard generator of the center of~$A[H_3]$ is $\delta=(s_1 s_2 s_3)^5$, hence $(\bar s_1 \bar s_2 \bar s_3)^5=1$. Let $\epsilon=\varphi_2(\bar s_1)=\varphi_2(\bar s_2)=\varphi_2(\bar s_3)\in \{\pm 1\}$. Then $1=\varphi_2(1)=\varphi_2((\bar s_1 \bar s_2 \bar s_3)^5)= \epsilon^{15}$, having that $\epsilon=1$.

\bigskip\noindent
Suppose that $\varphi_1$ is cyclic, that is, there is $w \in \SSS_5$ such that $\varphi_1(\bar s_1) = \varphi_1(\bar s_2) = \varphi_1(\bar s_3) = w$.
We denote by~$\ord (w)$ the order of~$w$.
As $w \in \SSS_5$, we have $\ord (w) \in \{1,2,3,4,5,6\}$.
On the other hand, as $(\bar s_1 \bar s_2 \bar s_3)^5=1$, we have $w^{15} = 1$, hence $\ord(w)$ divides $15$.
Thus, $\ord (w) \in \{1,3,5\}$.
Now, we let $\alpha = \bar s_2 \bar s_3^{-1}$ and $\beta = (\bar s_2 \bar s_3 \bar s_2)^5$.
Then, $\alpha, \beta \in \Ker (\varphi)$, $\alpha \neq 1$, $\alpha \beta \alpha \beta^{-1} = 1$, and $\Ker (\varphi)$ has generalized torsion. By \autoref{notinjective}, $\Phi$ is not injective.

\bigskip\noindent
Suppose that $\varphi_1$ is not cyclic.
Consider the two homomorphisms $\zeta_1, \zeta_2 : \overline{A [H_3]} \to \SSS_5$ defined by
\begin{gather*}
\zeta_1(\bar s_1) = (1,2,3,4,5)\,,\ 
\zeta_1(\bar s_2) = (1,4,2,3,5)\,,\ 
\zeta_1(\bar s_3) = (1,5,4,3,2)\,,\\
\zeta_2(\bar s_1) = (2,4)(3,5)\,,\ 
\zeta_2(\bar s_2) = (1,2)(4,5)\,,\
\zeta_2(\bar s_3) =(2,3)(4,5)\,.
\end{gather*}

\noindent
A direct computation with SageMath (see code in \citealp{code}) shows that every non-cyclic homomorphism from~$\overline{ A [H_3]}$ to~$\SSS_5$ is conjugate to either~$\zeta_1$ or~$\zeta_2$.
We can then suppose that $\varphi_1 \in \{ \zeta_1, \zeta_2 \}$.

\bigskip\noindent
If $\varphi_1 = \zeta_1$, by \citep[Proposition~9.4]{Boyland} and  \citep[Lemma~5.9]{BP} it follows that $\Phi(\bar s_1)$ is periodic or pseudo-Anosov.
If~$\Phi(\bar s_1)$ is periodic, then there is an integer $k \ge 1$ such that $\Phi(\bar s_1)^k = \id$, hence $\bar s_1^k$ is a non-trivial element of~$\Ker (\Phi)$ and~$\Phi$ is not injective.
Suppose that $\Phi(\bar s_1)$ is pseudo-Anosov.
As $\Phi((\bar s_1 \bar s_2)^5)$ is in the centralizer of~$\Phi (\bar s_1)$ in $\MM^*(\Sigma_{0,0}, \PP_5)$ and the centralizer of a  pseudo-Anosov element is virtually cyclic \citep[Lemma 8.13]{Ivano1}, there are integers $k,\ell \in \Z$, $\ell \neq 0$, such that $\Phi(\bar s_1)^k = \Phi((\bar s_1 \bar s_2)^5)^\ell$.
Let $\alpha = (\bar s_1 \bar s_2)^{5\ell} \bar s_1^{-k}$.
Then $\alpha$ is a non-trivial element of~$\Ker (\Phi)$ and~$\Phi$ is not injective.

\bigskip\noindent
Suppose that $\varphi_1 = \zeta_2$.
Then $\varphi_1(\bar s_1 \bar s_2) = (1, 4, 3, 5, 2)$, and again by \citep[Proposition~9.4]{Boyland} and \citep[Lemma~5.9]{BP}, $\Phi( \bar s_1 \bar s_2)$ is periodic or pseudo-Anosov.
If $\Phi( \bar s_1 \bar s_2)$ is periodic, there is an integer $k \ge 1$ such that $\Phi (\bar s_1 \bar s_2) ^k = \id$ and $\alpha = (\bar s_1 \bar s_2)^k$ is a non-trivial  element belonging to the kernel of~$\Phi$. This means that~$\Phi$ is not injective.
If $\Phi (\bar s_1 \bar s_2)$ is pseudo-Anosov,
then $\Phi ( (\bar s_1 \bar s_2)^5) = \Phi(\bar s_1 \bar s_2)^5$ is also pseudo-Anosov and $\Phi (\bar s_1)$ is in the centralizer of $\Phi ( (\bar s_1 \bar s_2)^5)$ in $\MM^* (\Sigma_{0,0}, \PP_5)$, which is virtually cyclic.  Hence there are integers $k, \ell \in \Z$, $\ell \neq 0$, such that $\Phi(\bar s_1)^\ell = \Phi((\bar s_1 \bar s_2)^5)^k$.
Let $\alpha = (\bar s_1 \bar s_2)^{5k} \bar s_1^{-\ell}$.
Therefore, $\alpha$ is a non-trivial element of~$\Ker (\Phi)$ and~$\Phi$ is not injective.
\end{proof}

\bigskip

\bibliography{Commensurability}

\bigskip\bigskip{\footnotesize%
\textit{ Mar\'{i}a Cumplido, IMB, UMR 5584, CNRS, Univ. Bourgogne Franche-Comt\'e, 21000 Dijon, France} \par
 \textit{E-mail address:} \texttt{\href{mailto:maria.cumplido.cabello@gmail.com}{maria.cumplido.cabello@gmail.com}}
 \medskip
 
\textit{ Luis Paris, IMB, UMR 5584, CNRS, Univ. Bourgogne Franche-Comt\'e, 21000 Dijon, France} \par
 \textit{E-mail address:} \texttt{\href{mailto:lparis@u-bourgogne.fr}{lparis@u-bourgogne.fr}} 
 
 }

\end{document}